\documentclass{amsart}
\usepackage{amssymb}
\usepackage{tikz-cd}
\usepackage{mathbbol} 
\usepackage{csquotes}

\usepackage[backend=biber,
	style=alphabetic,
	backref=true,
	url=false,
	doi=false,
	isbn=false,
	eprint=false
	]{biblatex}

\renewbibmacro{in:}{}  
\AtEveryBibitem{\clearfield{pages}}  

\DeclareFieldFormat{title}{\myhref{\mkbibemph{#1}}}
\DeclareFieldFormat
[article,inbook,incollection,inproceedings,patent,thesis,unpublished]
{title}{\myhref{\mkbibquote{#1\isdot}}}

\newcommand{\doiorurl}{%
	\iffieldundef{url}
	{\iffieldundef{eprint}
		{}
		{http://arxiv.org/abs/\strfield{eprint}}}
	{\strfield{url}}%
}

\newcommand{\myhref}[1]{%
	\ifboolexpr{%
		test {\ifhyperref}
		and
		not test {\iftoggle{bbx:eprint}}
		and
		not test {\iftoggle{bbx:url}}
	}
	{\href{\doiorurl}{#1}}
	{#1}%
}

\usepackage[bookmarks=true,
	linktocpage=true,
	bookmarksnumbered=true,
	breaklinks=true,
	pdfstartview=FitH,
	hyperfigures=false,
	plainpages=false,
	naturalnames=true,
	colorlinks=true,
	pagebackref=false,
	pdfpagelabels]{hyperref}

\hypersetup{
	colorlinks,
	citecolor=blue,
	filecolor=blue,
	linkcolor=blue,
	urlcolor=blue
}

\usepackage[capitalize, noabbrev]{cleveref}
\let\subsectionSymbol\S
\crefname{subsection}{\subsectionSymbol\!}{subsections}

\calclayout

\makeatletter
\@namedef{subjclassname@2020}{%
	\textup{2020} Mathematics Subject Classification}
\makeatother

\newtheorem{theorem}{Theorem}
\newtheorem{proposition}[theorem]{Proposition}
\newtheorem{lemma}[theorem]{Lemma}

\theoremstyle{definition}
\newtheorem{definition}[theorem]{Definition}
\newtheorem{example}[theorem]{Example}
\newtheorem{remark}[theorem]{Remark}

\newcommand{\M}{\cM}
\newcommand{\UM}{{\forget(\M)}}

\newcommand{\Z}{\mathbb{Z}}

\renewcommand{\S}{\mathbb{S}}
\newcommand{\Fp}{{\mathbb{F}_p}}

\newcommand{\Ch}{\mathsf{Ch}}

\DeclareMathOperator{\forget}{U}

\renewcommand{\th}{\mathrm{th}}

\DeclareMathOperator*{\colim}{colim}

\newcommand{\End}{\mathrm{End}}
\newcommand{\Hom}{\mathrm{Hom}}

\newcommand{\xra}[1]{\xrightarrow{#1}}

\newcommand{\cM}{\mathcal{M}}

\title{Cochain level May--Steenrod operations}

\author[R.~Kaufmann]{Ralph~M.~Kaufmann}
\address{R.K., Department of Mathematics, Department of Physics and Astronomy, Purdue University}
\email{\href{mailto:rkaufman@purdue.edu}{rkaufman@purdue.edu}}

\author{Anibal~M.~Medina-Mardones}
\address{A.M-M., Laboratory for Topology and Neuroscience at EPFL \and Max Planck Institute for Mathematics}
\email{\href{mailto:ammedmar@mpim-bonn.mpg.de}{ammedmar@mpim-bonn.mpg.de}}
\thanks{A.M-M. acknowledges financial support from Innosuisse grant \mbox{32875.1 IP-ICT-1}}



\usetikzlibrary{arrows,decorations.markings}
\tikzset{myptr/.style={decoration={markings,mark=at position 1 with %
			{\arrow[scale=2,>=stealth]{>}}},postaction={decorate}}}
		
\newsavebox\preproduct
\begin{lrbox}{\preproduct}
	\begin{tikzpicture}[scale=.3]
	\draw (0,0)--(0,-.8);
	\draw (0,0)--(.5,.5);
	\draw (0,0)--(-.5,.5);
	\end{tikzpicture} 
\end{lrbox}
\newcommand{\product}{
	\usebox\preproduct}

\newsavebox\precoproduct
	\begin{lrbox}{\precoproduct}
		\begin{tikzpicture}[scale=.3]
		\draw (0,0)--(0,.8);
		\draw (0,0)--(.5,-.5);
		\draw (0,0)--(-.5,-.5);
		\end{tikzpicture}
	\end{lrbox}
\newcommand{\coproduct}{
	\usebox\precoproduct}

\newsavebox\preboundary
\begin{lrbox}{\preboundary}
	\begin{tikzpicture}[scale=.3]
	\draw (0,0)--(0,1.3);
	\draw (.5,0)--(.5,1.3);
	\draw [fill] (0,0) circle [radius=0.1];
	\draw (.9,.6)--(1.3,.6);
	\draw (1.7,0)--(1.7,1.3);
	\draw (2.2,0)--(2.2,1.3);
	\draw [fill] (2.2,0) circle [radius=0.1];
	\end{tikzpicture}
\end{lrbox}
\newcommand{\boundary}{
	\usebox\preboundary}

\newsavebox\precoboundary
\begin{lrbox}{\precoboundary}
	\begin{tikzpicture}[scale=.3]
	\draw (0,0)--(0,1.3);
	\draw (.5,0)--(.5,1.3);
	\draw [fill] (0,1.3) circle [radius=0.1];
	\draw (.9,.6)--(1.3,.6);
	\draw (1.7,0)--(1.7,1.3);
	\draw (2.2,0)--(2.2,1.3);
	\draw [fill] (2.2,1.3) circle [radius=0.1];
	\end{tikzpicture}
\end{lrbox}

\newsavebox\precounit
\begin{lrbox}{\precounit}
	\begin{tikzpicture}[scale=.3]
	\draw (0,0)--(0,1.3);
	\draw [fill] (0,0) circle [radius=0.1];
	\end{tikzpicture}
\end{lrbox}
\newcommand{\counit}{
	\usebox\precounit}

\newsavebox\preidentity
\begin{lrbox}{\preidentity}
	\begin{tikzpicture}[scale=.3]
	\draw (0,0)--(0,1.3);
	\end{tikzpicture}
\end{lrbox}

\newsavebox\preunit
\begin{lrbox}{\preunit}
	\begin{tikzpicture}[scale=.3]
	\draw (0,0)--(0,1.3);
	\draw [fill] (0,1.3) circle [radius=0.1];
	\end{tikzpicture}
\end{lrbox}

\newsavebox\preassociativity
\begin{lrbox}{\preassociativity}
	\begin{tikzpicture}[scale=.2]
	\path[draw] (0,0)--(0,1)--(-1,2);
	\draw (0,1)--(1,2);
	\draw (-.5,1.5)--(0,2);
	\draw (1.2,1)--(1.8,1);
	\path[draw] (3,0)--(3,1)--(2,2);
	\draw (3,1)--(4,2);
	\draw (3.5,1.5)--(3,2);	
	\end{tikzpicture}
\end{lrbox}

\newsavebox\precoassociativity
\begin{lrbox}{\precoassociativity}
	\begin{tikzpicture}[scale=.2]
	\path[draw] (0,0)--(0,-1)--(-1,-2);
	\draw (0,-1)--(1,-2);
	\draw (-.5,-1.5)--(0,-2);
	\draw (1.2,-1)--(1.8,-1);
	\path[draw] (3,0)--(3,-1)--(2,-2);
	\draw (3,-1)--(4,-2);
	\draw (3.5,-1.5)--(3,-2);
	\end{tikzpicture}
\end{lrbox}

\newsavebox\preinvolution
\begin{lrbox}{\preinvolution}
	\begin{tikzpicture}[scale=.2]
	\path[draw] (0,0)--(0,.5)--(-.5,1)--(0,1.5)--(0,2);
	\path[draw] (0,.5)--(.5,1)--(0,1.5);
	\end{tikzpicture}
\end{lrbox}

\newsavebox\preleftcounitality
\begin{lrbox}{\preleftcounitality}
	\begin{tikzpicture}[scale=.3]
	\draw (0,0)--(0,.8);
	\draw (0,0)--(.5,-.5);
	\draw (0,0)--(-.5,-.5);
	\draw [fill] (-.5,-.5) circle [radius=0.1];
	\draw (.7,0)--(1.1,0);
	\path[draw] (1.5,-.5)--(1.5,.8);
	\end{tikzpicture}
\end{lrbox}
\newcommand{\leftcounitality}{
	\usebox\preleftcounitality}

\newsavebox\prerightcounitality
\begin{lrbox}{\prerightcounitality}
	\begin{tikzpicture}[scale=.3]
	\draw (0,0)--(0,.8);
	\draw (0,0)--(.5,-.5);
	\draw (0,0)--(-.5,-.5);
	\draw [fill] (.5,-.5) circle [radius=0.1];
	\draw (-.7,0)--(-1.1,0);
	\path[draw] (-1.5,-.5)--(-1.5,.8);
	\end{tikzpicture}
\end{lrbox}
\newcommand{\rightcounitality}{
	\usebox\prerightcounitality}

\newsavebox\preleftunitality
\begin{lrbox}{\preleftunitality}
	\begin{tikzpicture}[scale=.3]
	\draw (0,0)--(0,-.8);
	\draw (0,0)--(-.5,.5);
	\draw (0,0)--(.5,.5);
	\draw [fill] (.5,.5) circle [radius=0.1];
	\draw (.7,0)--(1.1,0);
	\path[draw] (1.5,.5)--(1.5,-.8);
	\end{tikzpicture}
\end{lrbox}

\newsavebox\prerightunitality
\begin{lrbox}{\prerightunitality}
	\begin{tikzpicture}[scale=.3]
	\draw (0,0)--(0,-.8);
	\draw (0,0)--(-.5,.5);
	\draw (0,0)--(.5,.5);
	\draw [fill] (-.5,.5) circle [radius=0.1];
	\draw (-.7,0)--(-1.1,0);
	\path[draw] (-1.5,.5)--(-1.5,-.8);
	\end{tikzpicture}
\end{lrbox}

\newsavebox\preproductcounit
\begin{lrbox}{\preproductcounit}
	\begin{tikzpicture}[scale=.3]
	\draw (0,0)--(0,-.8);
	\draw (0,0)--(.5,.5);
	\draw (0,0)--(-.5,.5);
	\draw [fill] (0,-.8) circle [radius=0.1];
	\end{tikzpicture}
\end{lrbox}
\newcommand{\productcounit}{
	\usebox\preproductcounit}

\newsavebox\preunitcoproduct
\begin{lrbox}{\preunitcoproduct}
	\begin{tikzpicture}[scale=.3]
	\draw (0,0)--(0,.8);
	\draw (0,0)--(.5,-.5);
	\draw (0,0)--(-.5,-.5);
	\draw [fill] (0,.8) circle [radius=0.1];
	\end{tikzpicture}
\end{lrbox}

\newsavebox\preleibniz
\begin{lrbox}{\preleibniz}
	\begin{tikzpicture}[scale=.245]
	\draw (0,.3)--(0,-.3);
	\draw (0,.3)--(.5,.8);
	\draw (0,.3)--(-.5,.8);
	\draw (0,-.3)--(0,.3);
	\draw (0,-.3)--(.5,-.8);
	\draw (0,-.3)--(-.5,-.8);
	
	\draw (.7,0)--(1.1,0);
	\draw (2.1,.8)--(2.1,.3)--(1.7,0)--(1.7,-.8);
	\draw (2.1,.3)--(2.9,-.3);
	\draw (3.3,.8)--(3.3,0)--(2.9,-.3)--(2.9,-.8);
	
	\draw (3.8,0)--(4.1,0);
	\draw (4.6,.8)--(4.6,0)--(5,-.3)--(5,-.8);
	\draw (5,-.3)--(5.8,.3);
	\draw (5.8,.8)--(5.8,.3)--(6.2,0)--(6.2,-.8);	
	\end{tikzpicture}
\end{lrbox}

\newsavebox\prebialgebra
\begin{lrbox}{\prebialgebra}
	\begin{tikzpicture}[scale=.5]
	\draw (0,.3)--(0,-.3);
	\draw (0,.3)--(.5,.8);
	\draw (0,.3)--(-.5,.8);
	\draw (0,-.3)--(0,.3);
	\draw (0,-.3)--(.5,-.8);
	\draw (0,-.3)--(-.5,-.8);
	
	\draw (.8,0)--(1.2,0);
	
	\draw (2.1,.8)--(2.1,.3)--(1.7,0)--(2.1,-.3)--(2.1,-.8);
	
	\draw (3.1,.8)--(3.1,.3)--(3.5,0)--(3.1,-.3)--(3.1,-.8);

	\draw (2.1,.3)--(3.1,-.3);
	\draw (2.1,-.3)--(2.5,-.06);
	\draw (3.1,.3)--(2.7,.06);	
	\end{tikzpicture}
\end{lrbox}

\newsavebox\precommutativity
\begin{lrbox}{\precommutativity}
	\begin{tikzpicture}[scale=.27]
	\draw (.3,0)--(.3,-1);
	\draw (.3,0)--(.8,.5);
	\draw (.3,0)--(-.2,.5);
	
	\draw (1.2,-.4)--(1.8,-.4);
	
	\draw (3,-.3)--(3,-1);
	\draw (2.5,0)--(3,-.3);
	\draw (3.5,0)--(3,-.3);
	\draw (2.46,0)--(2.9,.18);
	\draw (3.1,.3)--(3.5,.5);
	\draw (3.5,0)--(2.5,.5);
		\end{tikzpicture} 
	\end{lrbox}

\begin{document}
	\begin{abstract}
	Steenrod defined in 1947 the Steenrod squares on the mod 2 cohomology of spaces using explicit cochain formulae for the cup-$i$ products; a family of coherent homotopies derived from the broken symmetry of Alexander--Whitney's chain approximation to the diagonal. He later defined his homonymous operations for all primes using the homology of symmetric groups. This approach enhanced the conceptual understanding of the operations and allowed for many advances, but lacked the concreteness of their definition at the even prime. In recent years, thanks to the development of new applications of cohomology, the need to have an effectively computable definition of Steenrod operations has become a key issue. Using the operadic viewpoint of May, this article provides such definitions at all primes introducing multioperations that generalize the Steenrod cup-$i$ products on the simplicial and cubical cochains of spaces.
\end{abstract}
	\maketitle
	
\section{Introduction} \label{s:introduction}

The role Steenrod operations play in stable homotopy theory is hard to overstate.
The reason is that, given the representability of the cohomology functor, these operations together with the Bockstein homomorphism can be used to give a complete description of the algebraic structure naturally present on the mop-$p$ cohomology algebra of spaces.
For the even prime, Steenrod squares were introduced in \cite{steenrod1947products} via an explicit choice of coherent homotopical corrections to the broken symmetry of Alexander--Whitney's chain approximation to the diagonal, the so-called cup-$i$ products.
Later, for odd primes, their definition was given non-effectively using existence arguments based on the mod $p$ homology of symmetric groups \cite{steenrod1952reduced, steenrod1962cohomology, steenrod1962cohomology}.
This viewpoint enhanced the conceptual understanding of the operations and allowed for many advances \cite{adem1952iteration, milnor1958dual, adams1995stable}, but lacked the concreteness of their definition at the even prime.
The purpose of this paper is to fill this gap in the literature, introducing effective descriptions of multioperations at the cochain level -- generalizations of the Steenrod's cup-$i$ products -- that define Steenrod operations at all primes.

In recent years, thanks to the development of new applications of cohomology in areas such as applied topology and topological quantum field theories, the need to have an effectively computable definition of Steenrod operations has gained considerable importance.

In applied topology, the use of persistence homology \cite{edelsbrunner2002topological, carlsson2005barcode} has created many interdisciplinary research directions \cite{chan2013viral, hess2017cliques}, and the availability of formulae for the cup-$i$ products allowed for the development of a theory of persistence Steenrod modules accessing finer features of the data \cite{medina2018persistence}.

In physics, cup-$i$ products are used in the construction of actions for lattice theories whose fields are cochains \cite{gaiotto2016spin, bhardwaj2017fermionic, kapustin2017fermionic}.
This connection with physics motivated the development of effective versions of spin bordism \cite{brumfiel2016pontrjagin, brumfiel2018pontrjagin} prominently featuring higher derived structures at the cochain level \cite{medina2020cartan, medina2021adem}.

Following \cite{may1970general}, we take a more general approach to Steenrod operations that also includes Araki--Kudo--Dyer--Lashof operations on the mod $p$ homology of infinite loop spaces \cite{araki56squaring, dyer62lashof}.
We use the language of operads \cite{may1972geometry} to describe at the (co)chain level the integral structure required to define (co)homology operations at every prime.
We then describe effective constructions of this structure on three prominent models of the $E_\infty$-operad, identifying elements in them that represent Steenrod operations in the mod $p$ homology of their algebras; these are the Barratt--Eccles \cite{berger2004combinatorial}, surjection \cite{mcclure2003multivariable}, and $U(\mathcal M)$ \cite{medina2020prop1} operads.
Since the cochains of simplicial sets are equipped with effective and compatible algebra structures over each of these operads, we are able to explicitly describe canonical multioperations at the cochain level representing Steenrod operations at every prime, generalizing Steenrod's original cup-$i$ products \cite{steenrod1947products}.
An alternative approach based on the Eilenberg--Zilber contraction can be found in \cite{gonzalez2005cocyclic}.

The operad $\UM$ also acts effectively on cubical cochains \cite{medina2021cubical}, so we obtain explicit cochain level multioperations representing the Steenrod operations in this setting, which generalize the cup-$i$ constructions of Kadeishvili \cite{kadeishvili2003cupi} and Kr\v{c}\'{a}l--Pilarczy \cite{pilarczyk2016cubical}.
A context where the cubical viewpoint arises naturally is the study of base loop spaces.
This is through Baues' cubical generalization of Adams' cobar construction \cite{adams1956cobar, baues1998hopf}.
By using Baues' work, an application of the constructions presented in this paper is the explicit description, at the chain level, of Steenrod operations on the cobar construction of the coalgebra of chains on a reduced simplicial set \cite{medina2021cobar}.

Emphasizing their constructive nature, an implementation of all the constructions in this article can be found in the specialized computer algebra system \texttt{ComCH} \cite{medina2021computer}.

\subsection*{Outline}

We first introduce, in Section~\ref{s:preliminaries}, the conventions we will follow regarding chain complexes, simplicial sets and cubical sets.
Then, in Section~\ref{s:goup homology}, we review the key notions from group homology which we will use mainly for cyclic and symmetric groups.
Section~\ref{s:operads} is devoted to the language of operads and related structures, which we use in Section~\ref{s:steenrod} to introduce the notion of May--Steenrod structure, an integral structure at the (co)chain level inducing Steenrod operations for every prime.
Section~\ref{s:effective}, the bulk of this work, presents effective constructions of May--Steenrod structures on the Barratt--Eccles, surjection, and $U(\mathcal M)$ operads.
It also describes a natural $U(\mathcal M)$-algebra structure on the cochains of simplicial and cubical sets inducing natural May--Steenrod structures on them.
We end, in Section~\ref{s:outlook}, with an overview of some connections of this work to certain geometric and combinatorial structures and provide an outline of future research directions.
	\section*{Acknowledgement}

The authors thank Clemens Berger, Calista Bernard, Greg Brumfiel, Federico Cantero-Mor\'an, Greg Friedman, Kathryn Hess, Jens Kjaer, John Morgan, Andy Putman, Paolo Salvatore, Dev Sinha, and Dennis Sullivan for insightful discussions, and the anonymous referee for many keen observations and helpful suggestions.
	
\section{Preliminaries} \label{s:preliminaries}

\subsection{Chain complexes} \label{s s:chain cpx}

Let $R$ be a ring.
We denote by $(\mathbf{Ch}_R, \otimes, R)$ the symmetric monoidal category of homologically graded chain complexes of $R$-modules.
The set of $R$-linear maps between chain complexes as well as the tensor product of chain complexes are regarded as chain complexes in the usual way:
\begin{equation*}
\Hom(A, A^\prime)_n = \big\{f \ |\ a \in A_m \Rightarrow f(a) \in A^\prime_{m+n} \big\}, \qquad
\partial f = \partial \circ f - (-1)^{|f|}f \circ \partial,
\end{equation*}
\begin{equation*}
(A \otimes A^\prime)_n = \bigoplus_{p + q = n} A_p \otimes A^\prime_q, \qquad \partial (a \otimes a^\prime) = \partial a \otimes a^\prime + (-1)^{|a|} a \otimes \partial a^\prime.
\end{equation*}
We embed the category of $R$-modules as the full subcategory of $\mathbf{Ch}_R$ with objects concentrated in degree $0$.
The endofunctor $\Hom(-, R)$ is referred to as \textit{linear duality}.
We notice that if a chain complex is concentrated in non-negative degrees then its linear dual concentrates on non-positive ones.

The rings we will mostly be interested in are the group rings $\Z[\mathrm G]$ and $\mathbb{F}_p[\mathrm G]$ of finite groups, where $p$ is prime and $\mathbb{F}_p$ is the field with $p$ elements.

\subsection{Simplicial sets}

The \textit{simplex category} $\triangle$ is defined to have an object $[n] = \{0, \dots, n\}$ for every $n \in \mathbb{N}$ and a morphism $[m] \to [n]$ for each order-preserving function from $[m]$ to $[n]$.
The morphisms $\delta_i \colon [n-1] \to [n]$ and $\sigma_i \colon [n+1] \to [n]$ defined for $0 \leq i \leq n$ by
\begin{equation*}
\delta_i(k) =
\begin{cases} k & k < i, \\ k+1 & i \leq k, \end{cases}
\quad \text{ and } \quad
\sigma_i(k) =
\begin{cases} k & k \leq i, \\ k-1 & i < k, \end{cases}
\end{equation*}
generate all morphisms in the simplex category.

A \textit{simplicial set} $X$ is a contravariant functor from the simplex category to the category of sets, and a simplicial map is a natural transformation between two simplicial sets.
As is customary, we use the notation
\begin{equation*}
X\big( [n] \big) = X_n, \qquad X(\delta_i) = d_i, \qquad X(\sigma_i) = s_i,
\end{equation*}
and refer to elements in the image of any $s_i$ as \textit{degenerate}.

For each $n \in \mathbb{N}$, the simplicial set $\triangle^n$ is defined by
\begin{equation*}
\triangle^n_k = \Hom_{\triangle} \big( [k], [n] \big), \qquad
d_i(x) = x \circ \delta_i, \qquad
s_i(x) = x \circ \sigma_i,
\end{equation*}
and any simplicial set can be expressed as a colimit of these
\begin{equation*}
X \cong \colim_{\triangle^n \to X} \triangle^n.
\end{equation*}
We represent the non-degenerate elements of $\triangle^n_k$ as increasing sequences $[v_0, \dots, v_k]$ of non-negative integers each less than or equal to $n$.

The functor $N_\bullet$ of \textit{normalized chains} (with $R$-coefficients) is defined as follows:
\begin{equation*}
N_\bullet(X; R)_n = \frac{R \{ X_n \}}{R \{ s(X_{n-1}) \}}
\end{equation*}
where $s(X_{n-1}) = \bigcup_{i=0}^{n-1} s_i(X_{n-1})$, and $\partial_n \colon N_\bullet(X)_n \to N_\bullet(X)_{n-1}$ is given by
\begin{equation*}
\partial_n = \sum_{i=0}^{n} (-1)^id_{i}.
\end{equation*}
The functor of \textit{normalized cochains} $N^\bullet$ is defined by composing $N_\bullet$ with the linear duality functor $\Hom(-, R)$.

It is convenient to emphasize that
\begin{equation*}
N_\bullet(X; R) = \colim_{\triangle^n \to X} N_\bullet(\triangle^n; R).
\end{equation*}

\subsection{Cubical sets}

The \textit{cube category} $\square$ is the free strict monoidal category with a \textit{bipointed object}
\begin{equation*}
\begin{tikzcd}
1 \arrow[r, bend left, "\delta^0"] \arrow[r, bend right, "\delta^1"'] & 2 \arrow[r, "\sigma"] & 1
\end{tikzcd}
\end{equation*}
such that $\sigma \circ \delta^0 = \sigma \circ \delta^1 = \mathrm{id}$.
Explicitly, it contains an object $2^n$ for each non-negative integer $n$ and its morphisms are generated by the \textit{coface} and \textit{codegeneracy maps} defined by
\begin{align*}
\delta_i^\varepsilon & = \mathrm{id}_{2^{i-1}} \times \delta^\varepsilon \times \mathrm{id}_{2^{n-1-i}} \colon 2^{n-1} \to 2^n, \\
\sigma_i & = \mathrm{id}_{2^{i-1}} \times \, \sigma \times \mathrm{id}_{2^{n-i}} \colon 2^{n} \to 2^{n-1}.
\end{align*}

A \textit{cubical set} $X$ is a contravariant functor from the cube category to the category of sets, and a cubical map is a natural transformation between two cubical sets.
As is customary, we use the notation
\begin{equation*}
X\big( 2^n \big) = X_n \qquad X(\delta^\varepsilon_i) = d^\varepsilon_i \qquad X(\sigma_i) = s_i,
\end{equation*}
and refer to elements in the image of any $s_i$ as \textit{degenerate}.

For each $n \in \mathbb{N}$, the cubical set $\square^n$ is defined by
\begin{equation*}
\square^n_k = \Hom_{\square} \big( 2^k, 2^n \big), \qquad
d^\varepsilon_i(x) = x \circ \delta^\varepsilon_i, \qquad
s_i(x) = x \circ \sigma_i.
\end{equation*}
We represent the non-degenerate elements of $\square^n$ as sequences $x_1 \cdots\, x_n$ with each $x_i \in \big\{[0], [1], [0,1]\big\}$.
For example, $[0][01][1]$ represents $\delta^1 \times \mathrm{id} \times \delta^0$.
Any cubical set can be expressed as a colimit of these
\begin{equation*}
X \cong \colim_{\square^n \to X} \square^n.
\end{equation*}

The functor $N_\bullet$ of \textit{normalized chains} (with $R$-coefficients) is defined as follows: The chain complex $N_\bullet(\square^1)$ is simply the cellular chain complex of the interval,
isomorphic to
\begin{equation*}
\begin{tikzcd} [column sep = small, row sep = 0.1pt]
R\{[0], [1]\} & \arrow[l] R\{[0,1]\} \\
{[1] - [0]} & \arrow[l, |->] \left[0,1\right].
\end{tikzcd}
\end{equation*}
Set
\begin{equation*}
N_\bullet(\square^n; R) = N_\bullet(\square^1; R)^{\otimes n}
\end{equation*}
and define
\begin{equation*}
N_\bullet(X; R) = \colim_{\square^n \to X} N_\bullet(\square^n; R).
\end{equation*}

The functor of \textit{normalized cochains} $N^\bullet$ is defined by composing $N_\bullet$ with the linear duality functor $\Hom(-, R)$.
	
\section{Group homology} \label{s:goup homology}

Fixing notation, let $\mathrm{S}_r$ be the symmetric group of $r$ elements and let $\mathrm{C}_r$ be the cyclic group of order $r$ thought of as the subgroup of $\mathrm{S}_r$ generated by an element~$\rho$.
We denote this inclusion by $\iota \colon \mathrm C_r \to \mathrm S_r$.

A \textit{resolution} in $\mathbf{Ch}_R$ is a quasi-isomorphism $P \to M$ with each $P_r$ being a free $R$-module.
We will use the fact, explained in Section 6.5 of \cite{jacobson1989algebra}, that such resolutions exist for any chain complex $M$ concentrated in non-negative degrees.

Let $\mathrm G$ be a group and $M$ an $R[\mathrm G]$-module.
The \textit{homology of $\mathrm G$ with coefficients in $M$}, denoted by $H(\mathrm G; M)$, is defined as the homology of the chain complex $P \otimes_{R[\mathrm G]} M$ where $P \to R$ is any resolution in $\Ch_{R[\mathrm G]}$.
We will be particularly interested in the case when $M = \mathbb F_p(q)$ is the trivial or sign $\mathbb F_p[\mathrm{S}_r]$-module depending on if the parity of~$q$ is even or odd respectively.

We now review the group homology of finite cyclic groups.
For any ring $R$ the elements
\begin{equation} \label{eq: T and R definition}
\begin{split}
T &= \rho - 1, \\
N &= 1 + \rho + \cdots + \rho^{n-1},
\end{split}
\end{equation}
in $R[\mathrm{C}_r]$ generate the ideal of annihilators of each other.
Therefore, the chain complex of $R[\mathrm{C}_r]$-modules
\begin{equation} \label{eq: minimal resolution}
\begin{tikzcd} [column sep = .5cm]
\mathcal W(r) = R[\mathrm{C}_r]\{e_0\} & \arrow[l, "\,T"'] R[\mathrm{C}_r]\{e_1\} & \arrow[l, "\,N"'] R[\mathrm{C}_r]\{e_2\} & \arrow[l, "\,T"'] \cdots
\end{tikzcd}
\end{equation}
concentrated in non-negative degrees, with $\mathcal W(r)_d$ the free $R[\mathrm{C}_r]$-module $\mathcal W(r)_d$ generated by an element $e_d$, and differential induced from
\begin{equation*}
\partial(e_d) = \begin{cases}
Te_{d-1} & d \text{ odd,} \\
Ne_{d-1} & d \text{ even,}
\end{cases}
\end{equation*}
defines a resolution $\mathcal W(r) \to R$ in $\mathbf{Ch}_{R[\mathrm{C}_r]}$.

It follows from a straightforward computation that for any prime $p$ and integer $q$
\begin{equation*}
H_i(\mathrm{C}_p; \mathbb{F}_p(q)) = \mathbb{F}_p.
\end{equation*}

The homology of $\mathrm{S}_r$ is harder to compute.
With untwisted coefficients, the method of computation followed by several authors was to prove the injectivity of this homology into that of the infinite symmetric group and take advantage of a natural Hopf algebra structure on it.
A powerful result stemming from the deep connection of this question with infinite loop space theory is the existence of a homology isomorphism of spaces
\begin{equation*}
\Z \times B\mathrm{S}_\infty \to Q(S^0) = \Omega^\infty \Sigma^\infty (S^0)
\end{equation*}
credited to Dyer--Lashof \cite{dyer62lashof}, Barratt--Priddy and Quillen \cite{barratt1972priddyquillen}.

In this work, we are interested in the mod $p$ homology of $\mathrm{S}_p$ (with $p$ a prime) which, as explained in \cite[Corollary~VI.1.4]{adem2004milgram}, is detected by a group inclusion $\iota \colon \mathrm{C}_p \to \mathrm{S}_p$, that is, the map induced in mod $p$ group homology by $\iota$ is a surjection.
We now describe the kernel of this surjection.

\begin{lemma} \label{lem: Thom's theorem}
	Let $p$ be an odd prime and $q$ an integer.
	Consider
	\begin{equation*}
	(\iota_\ast)_d \colon H_d(\mathrm{C}_p; \mathbb{F}_p(q)) \to H_d(\mathrm{S}_p; \mathbb{F}_p(q))
	\end{equation*}
	then
	\begin{enumerate}
		\item If $q$ is even, $(\iota_\ast)_d = 0$ unless there is an integer $t$ so that $d = 2t(p-1)$ or $d = 2t(p-1) - 1$.
		\item If $q$ is odd, $(\iota_\ast)_d = 0$ unless there is an integer $t$ so that $d = (2t+1)(p-1)$ or $d = (2t+1)(p-1)-1$.
	\end{enumerate}
\end{lemma}

\begin{proof}
	This is proven as Theorem 4.1 in \cite{steenrod1953cyclic} where Thom is also credited with a different proof.
\end{proof}

In Section 5 we will see how the mod $p$ homology of symmetric groups defines operations on the mod $p$ homology of algebras that are commutative up to coherent homotopies.
Preparing for that, we first develop the language of $\Gamma$-modules, operads, and props.
	
\section{\texorpdfstring{$\Gamma$}{Gamma}-modules, operads and props} \label{s:operads}

In this section we set up a framework in which the structure responsible for Steenrod operations becomes most transparent.
Given our applications, we consider $\mathbf{Ch}_R$ as the base category, remarking that all definitions in this section apply to general closed symmetric monoidal categories.

\subsection{$\Gamma$-modules}
Recall that a group $\mathrm G$ can be thought of as a category with a single object and only invertible morphisms, and that a chain complex of left (resp. right) $R[\mathrm G]$-modules is the same as a covariant (resp. contravariant) functor from $\mathrm G$ to $\mathbf{Ch}_R$.
Taking inverses allows for the switch between left and right conventions.

A \textit{groupoid} is a small category where all morphisms are invertible.
\begin{definition}
	A \textit{$\Gamma$-module} is a covariant functor to $\mathbf{Ch}_R$ from a groupoid $\Gamma$ with objects being the natural numbers and morphisms satisfying $\Gamma(r,s) = \emptyset$ for $r \neq s$.
	We denote the category of $\Gamma$-modules and natural transformations by $\mathbf{Ch}_R^\Gamma$.
\end{definition}

We are mostly interested in two examples of $\Gamma$-modules, those associated to the groupoids $\mathrm{S}$ and $\mathrm{C}$ defined by
\begin{equation*}
\mathrm{S}(r, r) = \mathrm{S}_r, \qquad
\mathrm{C}(r,r) = \mathrm{C}_r,
\end{equation*}
for every $r \in \mathbb{N}$.
The inclusion $\mathrm{C}_r \to \mathrm{S}_r$ induces a forgetful functor
\begin{equation*}
\begin{tikzcd} [column sep = small]
\mathbf{Ch}_R^\mathrm{S} \arrow[r] & \mathbf{Ch}_R^\mathrm{C}.
\end{tikzcd}
\end{equation*}

Given an object $A$ in $\mathbf{Ch}_R$ there are two important $\Gamma$-modules associated to it; an $\mathrm{S}^{op}$-module known as \textit{endomorphism $\mathrm{S}^{op}$-module} $\End_A$, and an $\mathrm{S}$-module known as \textit{coendomorphism $\mathrm{S}$-module} $\End^A$.
These are defined by
\begin{align*}
\End_A(r) &= \Hom(A^{\otimes r},A), \\
\End^A(r) &= \Hom(A,A^{\otimes r}),
\end{align*}
with respective right and left actions defined by permutation of tensor factors.

Another groupoid of importance to us is $\mathrm{S} \times \mathrm{S}^{op}$ with covariant functors from it to $\mathbf{Ch}_R$ referred to as \textit{$\mathrm{S}$-bimodules}.
Notice that the inclusions $\mathrm{S} \to \mathrm{S} \times \mathrm{S}^{op}$ induced by $r \mapsto (r,1)$ and $r \mapsto (1,r)$ define forgetful functors
\begin{equation*}
\begin{tikzcd}[column sep=small, row sep=tiny]
& \mathbf{Ch}_R^{\mathrm{S} \times \mathrm{S}^{op}} \arrow[dl, "U_1"', out=-120, in=0] \arrow[dr, "U_2", out=-60, in=180] & \\
\mathbf{Ch}_R^{\mathrm{S}^{op}} & & \mathbf{Ch}_R^{\mathrm{S}}.
\end{tikzcd}
\end{equation*}
Explicitly, $U_1(\mathcal P)(r) = \mathcal P(r,1)$ and $U_2(\mathcal P)(r) = \mathcal P(1,r)$ for any $\mathcal P$ in $\mathbf{Ch}_R^{\mathrm{S} \times \mathrm{S}^{op}}$.
Notice that for any object $A$ in $\mathbf{Ch}_R$ the canonical \textit{endomorphism bimodule}
\begin{equation*}
\End_A^A(r, s) = \Hom(A^{\otimes r}, A^{\otimes s})
\end{equation*}
forgets via $U_1$ and $U_2$ to $\End_A$ and $\End^A$ respectively.

Using the groupoid automorphism sending every morphisms to its inverse we can identify $\Gamma$- and $\Gamma^{op}$-modules, and prove that the linear duality functor induces a morphism of $\mathrm{S}$-modules
\begin{equation*}
\End^A \to \End_{\Hom(A, R)}
\end{equation*}
for every object $A$ in $\mathbf{Ch}_R$.
We will use this identification freely in what follows.

A \textit{resolution} in $\mathbf{Ch}_R^\Gamma$ is a morphism $\phi$ of $\Gamma$-modules such that $\phi(r)$ is a resolution in the category of chain complexes of $R[\Gamma_r]$-modules for each $r \in \mathbb{N}$, where $\Gamma_r$ denotes $\Gamma(r,r)$.
A $\Gamma$-module $\mathcal R$ is said to be $E_\infty$ if $\mathcal R(0) = R$ and there exists a resolution $\mathcal R \to \underline{R}$ where $\underline{R}$ is the object in $\mathbf{Ch}_R^\Gamma$ defined by $\underline{R}(r) = R$ and $\underline{R}(\gamma) = \mathrm{id}_R$ for every $r \in \mathbb{N}$ and $\gamma \in \Gamma_r$.

We have the following evident generalization to the context of groupoids of the resolutions introduced in \eqref{eq: minimal resolution}.

\begin{definition} \label{def: minimal cyclic resolution}
	The \textit{minimal} $E_\infty$ $\mathrm{C}$-\textit{module} $\mathcal W$ is the functor in $\mathbf{Ch}_R^\mathrm{C}$ assigning to $r$ the chain complex
	\begin{equation*}
	\begin{tikzcd} [column sep = .5cm]
	\mathcal W(r) = R[\mathrm{C}_r]\{e_0\} & \arrow[l, "\,T"'] R[\mathrm{C}_r]\{e_1\} & \arrow[l, "\,N"'] R[\mathrm{C}_r]\{e_2\} & \arrow[l, "\,T"'] \cdots
	\end{tikzcd}
	\end{equation*}
	concentrated in non-negative degrees.
\end{definition}

\subsection{Operads and props}

Operads and props are respectively $\mathrm{S}$-modules and $\mathrm{S}$-bimodules enriched with further compositional structure.
These structures are best understood by abstracting the compositional structure naturally present in the endomorphism $\mathrm{S}$-module $\End_A$ (or $\End^A$), naturally an operad, and the endomorphism $\mathrm{S}$-bimodule $\End_A^A$, naturally a prop.

Succinctly, an operad $\mathcal O$ is an $\mathrm{S}$-module together with a collection of $R$-linear maps
\begin{equation*}
\mathcal O(r) \otimes \mathcal O(s) \to \mathcal O(r+s-1)
\end{equation*}
satisfying suitable associativity, equivariance and unitality conditions.
A prop $\mathcal P$ is an $\mathrm{S}$-bimodule together with two types of compositions; horizontal
\begin{equation*}
\mathcal P(r_1, s_1) \otimes \mathcal P(r_2, s_2) \to \mathcal P(r_1 + r_2, s_1 + s_2)
\end{equation*}
and vertical
\begin{equation*}
\mathcal P(r,s) \otimes \mathcal P(s, t) \to \mathcal P(r, t)
\end{equation*}
satisfying their own versions of associativity, equivariance and unitality.
For a complete presentation of these concepts we refer to \cite[Definitions 11 and 54]{markl2008props}.

We add that for any prop $\mathcal P$, the compositional structure of $\mathcal P$ defines an operad structure on $U_1(\mathcal P)$ and $U_2(\mathcal P)$.
We will use this automorphism without further notice when dealing with $\mathrm{S}^{op}$-modules.

We now introduce the type of operads that we are most interested in which, as we will discuss in the next section, are used to describe commutativity up to coherent homotopies.

\begin{definition} [\cite{may1972geometry}, \cite{boardman1973homotopy}] \label{def: e-infinity operad and prop}
	An operad is said to be an $E_\infty$-operad if its underlying $\mathrm{S}$-module is $E_\infty$, and a prop $\mathcal P$ is said to be an $E_\infty$-prop if either $U_1(\mathcal P)$ or $U_2(\mathcal P)$ is an $E_\infty$-operad.
\end{definition}

\subsection{Algebras, coalgebras and bialgebras}

A morphism of operads or of props is simply a morphism of their underlying $\mathrm{S}$-modules or $\mathrm{S}$-bimodules preserving the respective compositional structures.

Given a chain complex $A$, an operad $\mathcal O$ and a prop $\mathcal P$.
An $\mathcal O$-\textit{algebra} (resp. $\mathcal O$-\textit{coalgebra}) structure on $A$ is an operad morphism $\mathcal O \to \End_A$ (resp. $\mathcal O \to \End^A$), and a $\mathcal P$-\textit{bialgebra} structure on $A$ is a prop morphism $\mathcal P \to \End_A^A$.

We remark that the linear duality functor naturally transforms an $\mathcal O$-coalgebra structure on a chain complex into an $\mathcal O$-algebra structure on its dual.

Algebras over $E_\infty$-operads are the central objects of study in this work.
To develop intuition for them, let us consider a chain complex $A$ with an algebra structure over the constant functor $\underline{R}$, thought of as an operad with all compositions corresponding to the identity map $R \to R$.
The $\underline{R}$-algebra structure on $A$ is generated by a linear map $\mu \colon A \otimes A \to A$ which is (strictly) commutative and associative, and a linear map $\eta \colon R \to A$ that determines a (two-sided) unit for $\mu$.
Since $E_\infty$-operads are resolutions of $\underline{R}$, their algebras can be thought of as usual unital algebras where the commutativity and associativity relations hold up to coherent homotopies.
The two main examples to keep in mind are the cochains of spaces and the chains of infinite loop spaces.

\section{May--Steenrod structures} \label{s:steenrod}

We now introduce an operadic structure giving rise to Steenrod operations based in \cite{may1970general}.
In our presentation we emphasize the integral structure needed to define them at every prime.
For a more geometric treatment we refer the reader to \cite{may1972geometry, may76homology, lawson2020dyerlashof}, and for a different operadic approach at the even prime to \cite{chataur2005adem-cartan}.

Let us assume the ground ring to be $\Z$ unless stated otherwise.
\begin{definition} \label{def: May--Steenrod structure}
	A \textit{May--Steenrod structure} on an operad $\mathcal O$ is a
	morphism of $\mathrm{C}$-modules $\psi \colon \mathcal W \to \mathcal O$ for which there exists a factorization through an $E_\infty$-operad
	\begin{equation*}
	\begin{tikzcd}[column sep = normal, row sep = small]
	& \mathcal R \arrow[dr, "\phi", out=0] & \\
	\mathcal W \arrow[ur, "\iota", in=180] \arrow[rr, "\psi"] & & \mathcal O
	\end{tikzcd}
	\end{equation*}
	such that $\iota$ is a quasi-isomorphism and $\phi$ a morphism of operads.
\end{definition}

\begin{remark} \label{rmk: Deligne conjecture}
	In \cite{GerstenhaberVoronov} an operad morphism $\mathcal{A}ssoc \to \mathcal O$ is referred to as a \textit{multiplication} on $\mathcal O$.
	In this language, a choice of factorization $\phi \circ \iota$ of a May--Steenrod structure on $\mathcal O$ endows it with an $E_\infty$ \textit{multiplication} $\phi$.
\end{remark}

\begin{definition} \label{def: Steenrod products}
	Let $A$ be a chain complex.
	A May--Steenrod structure on $\mathrm{End}_A$ is referred to as one on $A$.
	Given one such structure $\psi \colon \mathcal W \to \End_A$, the \textit{Steenrod cup-}$(r, i)$ \textit{product} of $A$ is defined for every $r, i \geq 0$ as the image in $ \mathrm{End}(A^{\otimes r}, A)$ of $\psi(e_i)$.
\end{definition}

Let $A$ be equipped with a May--Steenrod structure
\begin{equation*}
\begin{tikzcd}[column sep = normal, row sep = small]
& \mathcal R \arrow[dr, "\phi", out=0] & \\
\mathcal W \arrow[ur, "\iota", out=55, in=180] \arrow[rr, "\psi"] & & \End_A.
\end{tikzcd}
\end{equation*}
We can relate this structure on $A$ to those considered by May in \cite{may1970general} as follows.
The morphism $\phi$ provides $A$ with the structure of a homotopy associative algebra defined by the image in $\Hom(A^{\otimes 2}, A)$ of a representative in $\mathcal R(2)$ of a generator of its $0^\th$-homology.
Restricting $\psi$ to arity $r$ defines a map $\theta \colon \mathcal W(r) \otimes A^{\otimes r} \to A$ that makes the pair $(A, \theta)$ into an object in May's category $\mathfrak{C}(\mathrm C_r, \infty, \Z)$ as presented in \cite[Definitions 2.1]{may1970general}.
Explicitly, this means that the pair is such that $\psi(e_0) \in \End_A(r)$ is $\mathrm C_r$-homotopic to the iterated product $A^{\otimes r} \to A$; a claim that follows from the iterated product being a representative of a generator of the $0^\th$-homology of $\mathcal R(r)$, and $\iota$ being a quasi-isomorphism of $\mathrm C$-modules.
Furthermore, for $r$ equal to a prime $p$, tensoring the integers with $\Fp$ makes the pair $(A, \theta)$ into an object in May's category $\mathfrak{C}(\mathrm C_p, \infty, \mathbb{F}_p)$.
For any object $(A, \theta)$ in this category, Definition~2.2 in \cite{may1970general} defines operations on the mod $p$ homology of $A$, a construction we review below.
In particular, if $A$ is given by the cochains of a space these products agree with Steenrod's original definitions, and for $A$ being the chains on an infinite loop space, with those defined by Araki--Kudo and Dyer--Lashof.

For the rest of this section $A$ denotes a chain complex equipped with a May--Steenrod structure.

\begin{definition}
	For any prime $p$, the $\mathbb{F}_p$-linear map
	\begin{equation*}
	D^p_i \colon (A \otimes \mathbb{F}_p) \to (A \otimes \mathbb{F}_p)
	\end{equation*}
	is defined by sending $a$ to the Steenrod cup-$(p, i)$ product of $(a \otimes \cdots \otimes a) \in (A \otimes \mathbb{F}_p)^{\otimes p}$ if $i \geq 0$ and to $0$ otherwise.
\end{definition}

We notice that if $a$ is of degree $q$ then $D^p_i(a)$ is of degree $q + (p-1)q + i$.

\begin{definition}
	For any integer $s$, the \textit{Steenrod operation}
	\begin{equation*}
	P_s \colon H_\bullet(A; \mathbb{F}_2) \to H_{\bullet + s}(A; \mathbb{F}_2)
	\end{equation*}
	is defined by sending the class represented by a cycle $a \in (A \otimes \mathbb{F}_2)$ of degree $q$ to the class represented by $D^2_{s-q}(a)$.
\end{definition}

Notice that the Steenrod operations above, corresponding to Steenrod squares in the context of spaces, are determined by the Steenrod cup-$(2,i)$ products with $\mathbb{F}_2$-coefficients.
These binary operations are known as cup-$i$ products \cite{steenrod1947products, medina2021newformulas} in the space context.
In a similar way, the operations $P$ and $\beta P$ defined below for odd primes are determined by the Steenrod cup-$\big(p, k(p-1)-\varepsilon\big)$ products for $\varepsilon \in \{0,1\}$.
We can explain the appearance of these specific Steenrod cup-$(p,i)$ products as follows.
The increase on the degree of a $q$-cycle after applying $D^p_{k(p-1)-\varepsilon}$ to it is $(p-1)(q+k) - \varepsilon$, which can be rewritten as $2t(p-1) - \varepsilon$ if $q$ is even, and $(2t+1)(p-1) - \varepsilon$ if $q$ is odd.
According to Lemma \ref{lem: Thom's theorem}, these are the only homologically non-trivial cases.

\begin{definition} \label{def: Steenrod operations at odd prime}
	For any integer $s$, the \textit{Steenrod operations}
	\begin{equation*}
	P_s \colon H_\bullet(A; \mathbb{F}_p) \to H_{\bullet + 2s(p-1)}(A; \mathbb{F}_p)
	\end{equation*}
	and
	\begin{equation*}
	\beta P_s \colon H_\bullet(A; \mathbb{F}_p) \to H_{\bullet + 2s(p-1) - 1}(A; \mathbb{F}_p)
	\end{equation*}
	are defined by sending the class represented by a cycle $a \in (A \otimes \mathbb{F}_p)$ of degree $q$ to the classes represented respectively for $\varepsilon \in\{0,1\}$ by
	\begin{equation*}
	(-1)^s \nu(q) D^p_{(2s-q)(p-1)-\varepsilon}(a)
	\end{equation*}
	where $\nu(q) = (-1)^{q(q-1)m/2}(m!)^q$ and $m = (p-1)/2$.
\end{definition}

\begin{remark}
	The use of the coefficient function $\nu(q)$ is motivated by the identity $D_{q(p-1)}^p(a) = \nu(q)a$ in the case of spaces (see \cite[(6.1)]{steenrod1953cyclic}).
	The notation $\beta P_s$ is motivated by the relationship of this operator and the Bockstein of the reduction $\mathbb Z \to \mathbb Z/p\mathbb Z$.
\end{remark}

Steenrod operations defined as above satisfy the so-called \textit{Adem relations}.
Below we present its most common form and refer to Theorem 4.7 in \cite{may1970general} for a complete list.

\begin{lemma}
	Let $A$ be equipped with a May--Steenrod structure.
	Then,
	\begin{enumerate}
		\item If $p = 2$ and $a > 2b$, then
		\begin{equation*}
		P_{a} P_{b} = \sum_i \binom{2i-a}{a-b-i-1} P_{a+b-i}P_i,
		\end{equation*}
		\item If $p > 2$ and $a > pb$, then
		\begin{equation*}
		P_{a} P_{b} = \sum_i (-1)^{a+i} \binom{pi-a}{a-(p-1)b-i-1} P_{a+b-i}P_i.
		\end{equation*}
	\end{enumerate}
\end{lemma}

\begin{proof}
	As described after Definition~\ref{def: Steenrod products}, for any prime $p$ the pair $(A, \theta)$ is an object in May's category $\mathfrak{C}(\mathrm C_p, \infty, \mathbb{F}_p)$.
	Furthermore, since we are demanding a factorization $\phi \circ \iota$ with $\phi$ being an operad map from an $E_\infty$-operad to $\End_A$, the pair $(A, \theta)$ is an \textit{Adem object} in the sense of \cite[Definition 4.1]{may1970general} and the statement presented here is stated and proven as part of \cite[Theorem 4.7]{may1970general}.
\end{proof}

So far we have considered $\mathrm{C}$-modules, operads and related structures over the category of chain complexes.
It is also useful to consider them over the category of coalgebras, that is to say requiring each chain complex to be equipped with a coproduct and all structure maps to be morphisms of coalgebras.
As described in Definition~1.2 of \cite{may1970general}, the $\mathrm{C}$-module $\mathcal W$ lifts to this category.
A \textit{comultiplicative May--Steenrod structure} on an operad $\mathcal O$ is a morphism of $\mathrm{C}$-modules $\psi \colon \mathcal W \to \mathcal O$ for which there exists a factorization through an $E_\infty$-operad over the category of coalgebras
\begin{equation*}
\begin{tikzcd}[column sep = normal, row sep = small]
& \mathcal R \arrow[dr, "\phi", out=0] & \\
\mathcal W \arrow[ur, "\iota", out=55, in=180] \arrow[rr, "\psi"] & & \mathcal O
\end{tikzcd}
\end{equation*}
such that $\iota$ is a quasi-isomorphism over the category of coalgebras and $\phi$ is a morphism of operads.

Chain complexes equipped with a comultiplicative May--Steenrod structure satisfy the so-called \textit{Cartan relations}.

\begin{lemma}
	Let $A$ be equipped with a comultiplicative May--Steenrod structure.
	For any two mod $p$ homology classes $[\alpha]$ and $[\beta]$ we have
	\begin{equation*}
	P_s\big([\alpha] [\beta]\big) = \sum_{i+j=s} P_i\big( [\alpha] \big) P_j\big( [\beta] \big),
	\end{equation*}
\end{lemma}

\begin{proof}
	As described after Definition~\ref{def: Steenrod products}, for any prime $p$ the pair $(A, \theta)$ is an object in May's category $\mathfrak{C}(\mathrm C_p, \infty, \mathbb{F}_p)$.
	Furthermore, since we are demanding a factorization $\phi \circ \iota$ with $\iota$ being a quasi-isomorphism in the category of $\mathrm{C}$-modules over the category of coalgebras, this is a \textit{Cartan object} as defined in \cite[p.161]{may1970general}.
	The statement presented here is stated in page~165 loc.~cit.
\end{proof}

For the even prime, effective proofs at the cochain level of the Adem and Cartan relations have been given respectively in \cite{medina2021adem} and \cite{medina2020cartan}.
Explicitly, these construct cochains whose coboundaries descends to the relations in cohomology.
	
\section{Effective constructions} \label{s:effective}

In this section we construct explicit May--Steenrod structures on three well know combinatorial $E_\infty$-operads: the Barratt--Eccles operad $\mathcal E$ (see \cite{berger2004combinatorial}), the surjection operad $\mathcal X$ (see \cite{mcclure2003multivariable}), and the operad $U(\mathcal M)$ associated to the finitely presented $E_\infty$-prop $\mathcal M$ introduced in \cite{medina2020prop1}.
We also define a natural and effective May--Steenrod structures on the normalized cochains of any simplicial or cubical set using that these are algebras over the operad $U(\mathcal M)$.

Figure~\ref{fig: bigsummary} presents a diagrammatical representation of the constructions in this section.

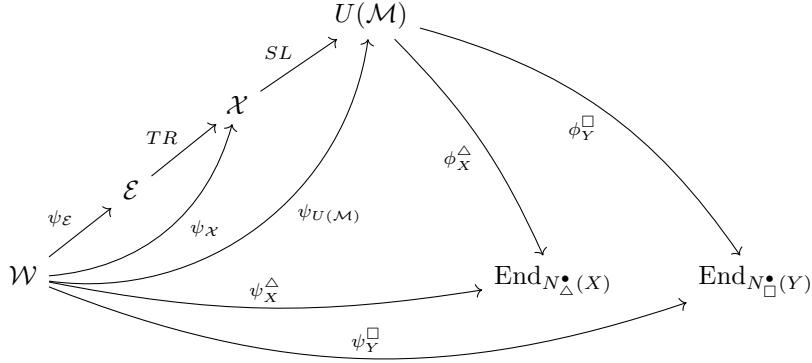
\begin{figure}[ht]
	\begin{tikzcd}
	& & & U(\mathcal M) \arrow[dddr, bend left=10, "\phi^\triangle_X"'] \arrow[dddrr, bend left=20, "\phi^\square_Y"'] & & \\
	& & \mathcal X \arrow[ru, "SL"] & & & \\
	& \mathcal E \arrow[ur, "TR"] & & & & \\
	\mathcal W \arrow[rrrr, "\psi^\triangle_X", bend right=10] \arrow[rrrrr, "\psi^\square_Y", bend right=20] \arrow[ur, "\psi_{\mathcal E}"] \arrow[uurr, "\psi_{\mathcal X}"', bend right=35, pos=.6] \arrow[uuurrr, "\psi_{U(\mathcal M)}"', bend right=45, pos=0.6] & & & & \End_{N^\bullet_\triangle(X)} & \End_{N^\bullet_\square(Y)}
	\end{tikzcd}
	\caption{Summary of effective constructions: May--Steenrod structures on the Barratt--Eccles $\mathcal E$, surjection $\mathcal X$, and $U(\mathcal M)$ operads, and natural May--Steenrod structures on the normalized chains of a simplicial or cubical set. We remark that the maps $TR$ and $SL$ require different sign conventions.}
	\label{fig: bigsummary}
\end{figure}

\subsection{Barratt--Eccles operad}

In this subsection we effectively describe a May--Steenrod structure on the Barratt--Eccles operad via explicit formulae.

We begin by reviewing the $\mathrm{S}$-module structure underlying the Barratt--Eccles operad and, since we will not use it in this work, refer to \cite{berger2004combinatorial} for a description of if composition structure.
For a non-negative integer $r$ define the simplicial set $E(\mathrm S_r)$ by
\begin{equation} \label{eq: milnor model of symmetric}
\begin{split}
E(\mathrm S_r)_n &= \{ (\sigma_0, \dots, \sigma_n)\ |\ \sigma_i \in \mathrm{S}_r\}, \\
d_i(\sigma_0, \dots, \sigma_n) &= (\sigma_0, \dots, \widehat{\sigma}_i, \dots, \sigma_n), \\
s_i(\sigma_0, \dots, \sigma_n) &= (\sigma_0, \dots, \sigma_i, \sigma_i, \dots, \sigma_n) \\
\end{split}
\end{equation}
with a left $\mathrm S_r$-action given by
\begin{equation*}
\sigma (\sigma_0, \dots, \sigma_n) = (\sigma \sigma_0, \dots, \sigma \sigma_n).
\end{equation*}
The chain complex resulting from applying the functor of normalized integral chains to it is the arity $r$ part of the Barratt--Eccles operad $\mathcal E$.

\begin{definition} \label{def: Steenrod-Adem on Barratt--Eccles}
	For every $r \geq 0$, let $\psi_{\mathcal E}(r) \colon \mathcal W(r) \to \mathcal E(r)$ be the $\Z[\mathrm{C}_r]$-linear map defined on basis elements by
	\begin{equation*}
	\psi_{\mathcal E}(r)(e_{n}) = \begin{cases}
	\displaystyle{\sum_{r_1, \dots, r_m}} \big(\rho^0, \rho^{r_1}, \rho^{r_1+1}, \rho^{r_2}, \dots, \rho^{r_{m}}, \rho^{r_{m}+1} \big) & n = 2m, \\
	\displaystyle{\sum_{r_1, \dots, r_m}} \big(\rho^0, \rho^1, \rho^{r_1}, \rho^{r_1+1}, \dots, \rho^{r_{m}}, \rho^{r_{m}+1} \big) & n = 2m+1,
	\end{cases}
	\end{equation*}
	where the sum is over all $r_1, \dots, r_m \in \{0, \dots, r-1\}$.
\end{definition}

\begin{theorem} \label{thm: Steenrod-Adem on Barratt--Eccles}
	The morphism of $\mathrm{C}$-modules
	\begin{equation*}
	\psi_{\mathcal E} \colon \mathcal W \to \mathcal E
	\end{equation*}
	defines a May--Steenrod structure on the Barratt--Eccles operad.
\end{theorem}

\begin{proof}
	Since $\mathcal E$ is an $E_\infty$-operad, we simply need to prove that the $\Z[\mathrm{C}_r]$-linear map
	\begin{equation*}
	\psi_{\mathcal E}(r) \colon \mathcal W(r) \to \mathcal E(r)
	\end{equation*}
	is a quasi-isomorphism for every $r \geq 0$.
	We simplify notation and write $\psi$ instead of $\psi_{\mathcal E}(r)$.
	To show $\psi$ is a chain map we proceed by induction.
	Notice that
	\begin{equation*}
	\psi(\partial e_0) = 0 = \partial \psi(e_0)
	\end{equation*}
	and assume $\psi(\partial e_{k-1}) = \partial \psi(e_{k-1})$.
	If $k = 2n$ we have
	\begin{align*}
	\partial \psi(e_{2n}) & =
	\partial \sum_{r_1, \dots, r_n}
	\big(\rho^0, \rho^{r_1}, \rho^{r_1+1}, \dots, \rho^{r_n}, \rho^{r_n+1} \big) \\ & =
	\partial \sum_{r_2, \dots, r_n} \sum_{r_1 = 0}^{p-1}
	\big(\rho^0, \rho^{r_1} \, (\rho^0, \rho^{1}, \dots, \rho^{r_n-r_1}, \rho^{r_n - r_1 +1}) \big) \\ & =
	\partial \sum_{r_2, \dots, r_n}
	\big(\rho^0, N\, (\rho^{0}, \rho^{1}, \dots, \rho^{r_n}, \rho^{r_n + 1}) \big) \\ & =
	\sum_{r_2, \dots, r_n}
	N\, \big( \rho^{0}, \rho^{1}, \dots, \rho^{r_n}, \rho^{r_n + 1} \big) \\ & -
	\sum_{r_2, \dots, r_n}
	\big(\rho^0, \partial \, N \, (\rho^{0}, \rho^{1}, \dots, \rho^{r_n}, \rho^{r_n+1}) \big) \\ & =
	N \psi(e_{2n-1}) - (\rho^0, \partial N \psi (e_{2n-1})) \\ & =
	\psi(N e_{2n-1}) - (\rho^0, \psi (\partial N e_{2n-1})) \\ & =
	\psi(\partial e_{2n}) - (\rho^0, \psi (\partial^2 e_{2n})) \\ & =
	\psi(\partial e_{2n}).
	\end{align*}
	If $k = 2n+1$ we have
	\begin{align*}
	\partial \psi(e_{2n+1}) & =
	\partial \sum_{r_1, \dots, r_n}
	\big(\rho^0, \rho^1, \rho^{r_1}, \rho^{r_1+1}, \dots, \rho^{r_n}, \rho^{r_n+1} \big) \\ & =
	\partial \sum_{r_1, \dots, r_n}
	\big(\rho^0, \rho^{1} \, (\rho^0, \rho^{r_1-1}, \rho^{r_1}, \dots, \rho^{r_n - 1}, \rho^{r_n}) \big) \\ & =
	\partial \sum_{r_1, \dots, r_n}
	\big(\rho^0, T\, (\rho^0, \rho^{r_1-1}, \rho^{r_1}, \dots, \rho^{r_n - 1}, \rho^{r_n}) \big) \\ & =
	\sum_{r_1, \dots, r_n}
	T \, \big( \rho^0, \rho^{r_1-1}, \rho^{r_1}, \dots, \rho^{r_n - 1}, \rho^{r_n} \big) \\ & -
	\sum_{r_1, \dots, r_n}
	\big(\rho^0, \partial \, T \, (\rho^0, \rho^{r_1-1}, \rho^{r_1}, \dots, \rho^{r_n - 1}, \rho^{r_n}) \big) \\ & =
	T \psi(e_{2n}) - (\rho^0, \partial T \psi (e_{2n})) \\ & =
	\psi(T e_{2n}) - (\rho^0, \psi (\partial T e_{2n})) \\ & =
	\psi(\partial e_{2n+1}) - (\rho^0, \psi (\partial^2 e_{2n+1})) \\ & =
	\psi(\partial e_{2n+1})
	\end{align*}
	where for the third equality we used that for any $r_1, \dots, r_n$
	\begin{equation*}
	(\rho^0, \rho^0, \rho^{r_1-1}, \rho^{r_1}, \dots, \rho^{r_n - 1}, \rho^{r_n}) = 0.
	\end{equation*}
	This map is a quasi-isomorphism since both complexes have the homology of a point and $\psi(e_0)$ represents a generator of the homology.
\end{proof}

\begin{example}
	Table \ref{fig: Barratt--Eccles Steenrod products} shows $\psi_{\mathcal E}(r)(e_n)$ for small values of $r$ and $n$.
\end{example}

\begin{table}[ht]
	\centering
	\resizebox{0.8\columnwidth}{!}{%
		\renewcommand{\arraystretch}{1.2}
		\begin{tabular}{|c||c|c|c|}
			\hline
			$r$ & $n=2$ & $n=3$ & $n=4$ \\
			\hline
			2 & (0,1,0) & (0,1,0,1) & (0,1,0,1,0) \\
			\hline
			3 & (0,1,2) + (0,2,0) & (0,1,2,0) + (0,1,0,1) & \phantom{+} (0,1,2,0,1) + (0,1,2,1,2) \\
			& & & + (0,2,0,1,2) + (0,2,0,2,0) \\
			\hline
			4 & (0,1,2) + (0,2,3) & (0,1,2,3) + (0,1,3,0) & \phantom{+} (0,1,2,3,0) + (0,1,2,0,1) \\
			& + (0,3,0) & + (0,1,0,1) &
			+ (0,1,2,1,2) + (0,2,3,0,1) \\
			& & & + (0,2,3,1,2) + (0,2,3,2,3) \\
			& & & + (0,3,0,1,2) + (0,3,0,2,3) \\
			& & & + (0,3,0,3,0) \\
			\hline
		\end{tabular}
	}
	\vspace*{3pt}
	\caption{The elements $\psi_{\mathcal E}(r)(e_n)$ for small values of $r$ and $n$ where we are denoting $(\rho^{r_0}, \dots, \rho^{r_n})$ simply by $(r_0, \dots, r_n)$.}
	\label{fig: Barratt--Eccles Steenrod products}
\end{table}

\begin{remark}
	The natural construction \eqref{eq: milnor model of symmetric} is defined for any group, in particular, for finite cyclic groups, and the inclusion $\iota \colon \mathrm C_r \to \mathrm S_r$ induces both: a simplicial inclusion $E(\iota) \colon E(\mathrm C_r) \to E(\mathrm S_r)$ and one of $\mathrm C$-modules
	\begin{equation*}
	N_\bullet E(\iota) \colon N_\bullet E(\mathrm C) \to N_\bullet E(\mathrm S) = \mathcal E.
	\end{equation*}
	We remark that the image of our map $\psi_{\mathcal E}$ lies in the subcomplex $N_\bullet E(\mathrm C)$, so the map $\psi_{\mathcal E}$ factors as
	\begin{equation*}
	\psi_{\mathcal E} \colon
	\begin{tikzcd}
	\mathcal W \arrow[r] &[-3pt] N_\bullet E(\mathrm C) \arrow[r, "N_\bullet E(\iota)\,"] &[5pt] \mathcal E.
	\end{tikzcd}
	\end{equation*}
\end{remark}

\subsection{Surjection operad}

In this subsection we introduce a May--Steenrod structure on the surjection operad.
There are two widely used sign conventions for this operad respectively due to McClure--Smith \cite{mcclure2003multivariable} and Berger--Fresse \cite{berger2004combinatorial}.
Using the Berger--Fresse convention, we can define a May--Steenrod structure on the surjection operad by composing the map $\psi_{\mathcal E}$ with the \textit{table reduction} quasi-isomorphism $\mathcal E \to \mathcal X$ introduced in Section 1.3 of \cite{berger2004combinatorial}.
We define a May--Steenrod structure on the surjection operad in a convention independent way which recovers the table reduction May--Steenrod structure in the Berger--Fresse case.

Let us start by recalling the definition of the surjection operad.
For a non-negative integer $r$ let $\mathcal X(r)_n$ be the free abelian group generated by all functions from $\{1, \dots, n+r\}$ to $\{1, \dots, r\}$ modulo the subgroup generated by degenerate functions, i.e., those which are either non-surjective or have a pair of equal consecutive values.
We only describe the McClure--Smith convention since we refer to it in subsequent sections.
The boundary map and symmetric action in this case is defined using the Koszul convention regarding a surjection $s \colon \{1, \dots, n+r\} \to \{1, \dots, r\} $ as the top-dimensional generator in the chain complex
\begin{equation} \label{eq: surjection boundary}
\bigotimes_{i = 1}^r N_\bullet \big( \triangle^{s^{-1}(i)}; \Z \big).
\end{equation}
Explicitly, if we think of $s$ as a sequence of integers $\big( s(1), \dots, s(n+r) \big)$ the boundary of $s$ is the sum of sequences obtained by orderly removing one at a time the occurrences of $1$ with alternating signs, then those of $2$ with starting sign equal to that of the last removal of $1$, and so on.
Since we do not use the composition structure we refer to \cite{mcclure2003multivariable} for it.

Let us recall the chain contraction of $\mathcal X(r)$ onto $\mathcal X(r-1)$ used by McClure--Smith to prove that $\mathcal X$ is an $E_\infty$-operad and whose introduction is credited to Benson \cite{benson1998rep}.
Let the maps $i$, $p$, and $s$ be defined on basis elements, represented by sequences, as follows: $i \colon \mathcal X(r-1) \to \mathcal X(r)$ places a $1$ at the beginning of the sequence and increases each of the original entries by $1$, $p \colon \mathcal X(r) \to \mathcal X(r-1)$ takes the sequence to $0$ unless it contains a single occurrence of $1$, in which case $p$ removes the $1$ and decreases each of the remaining entries by $1$, and finally, $s \colon \mathcal X(r) \to \mathcal X(r)$ places a 1 at the beginning of the sequence; if the sequence already begins
with a 1, then the new sequence is degenerate so $s$ takes it to $0$.
These maps satisfy
\begin{equation*}
p i = \mathrm{id}
\qquad \text{ and } \qquad
\mathrm{id} - i p = \partial s + s \partial.
\end{equation*}
The compositions $i^{r-1}$ and $s^{r-1}$ define a contraction of $\mathcal X(r)$ onto $\mathcal X(1) \cong \Z$ with homotopy
\begin{equation*}
h = s + i\, s\, p + \cdots + i^{r-1} s\, p^{r-1},
\end{equation*}
i.e., they satisfy
\begin{equation*}
p^{r-1} i^{r-1} = \mathrm{id}
\qquad \text{ and } \qquad
\partial h + h \partial = \mathrm{id} - i^{r-1}\,p^{r-1}.
\end{equation*}

\begin{definition} \label{def: Steenrod-Adem on surjection}
	For every $r \geq 0$, let $\psi_{\mathcal X}(r) \colon \mathcal W(r) \to \mathcal X(r)$ be the $\Z[\mathrm{C}_r]$-linear map defined recursively on basis elements by
	\begin{align*}
	\psi_{\mathcal X}(r)(e_0) & = (1, \dots, r), \\
	\psi_{\mathcal X}(r)(e_{2m+1}) & = h\,T\,\psi_{\mathcal X}(r)(e_{2m}), \\
	\psi_{\mathcal X}(r)(e_{2m}) & = h\,N\,\psi_{\mathcal X}(r)(e_{2m-1}),
	\end{align*}
	where $T$ and $N$ are defined in \eqref{eq: T and R definition}.
\end{definition}

\begin{theorem} \label{thm: Steenrod-Adem on surjection MS convention}
	The morphism of $\mathrm{C}$-modules
	\begin{equation*}
	\psi_{\mathcal X} \colon \mathcal W \to \mathcal X
	\end{equation*}
	defines a May--Steenrod structure on the Surjection operad.
\end{theorem}

\begin{proof}
	Since $\mathcal X$ is an $E_\infty$-operad, we simply need to prove that the $\Z[\mathrm{C}_r]$-linear map
	\begin{equation*}
	\psi_{\mathcal X}(r) \colon \mathcal W(r) \to \mathcal X(r)
	\end{equation*}
	introduced in Definition \ref{def: Steenrod-Adem on surjection} is a quasi-isomorphism for every $r \geq 0$.
	We simplify notation and write $\psi$ instead of $\psi_{\mathcal X}(r)$.
	To show that $\psi$ is a chain map we proceed by induction.
	Notice that
	\begin{equation*}
	\psi(\partial e_0) = 0 = \partial \psi(e_0)
	\end{equation*}
	and assume $\psi(\partial e_{n-1}) = \partial \psi(e_{n-1})$.
	For $n = 2m+1$ we have
	\begin{align*}
	\partial \psi(e_{2m+1})
	& =
	\partial\, h\, T\, \psi(e_{2m}) \\
	& =
	T\, \psi(e_{2m}) - i^{r-1} p^{r-1}\, \psi(e_{2m}) -
	h\, \partial\, T\, \psi(e_{2m}) \\
	& =
	T\, \psi(e_{2m}) -
	h\, T\, \psi(\partial\, e_{2m}) \\
	& =
	T\, \psi(e_{2m}) -
	h\, \psi(T\,N\, e_{2m-1}) \\
	& =
	T \psi(e_{2m}).
	\end{align*}
	For $n = 2m$ the proof is analogous.
	The chain map $\psi$ is a quasi-isomorphism since both complexes have the homology of a point and $\psi(e_0) = (1, \dots, r)$ represents a generator of the homology.
\end{proof}

For the rest of this paper we use the McClure--Smith sign convention on $\mathcal X$.

\begin{example}
	Table \ref{fig: surjection Steenrod products} shows $\psi_{\mathcal X}(r)(e_n)$ for small values of $r$ and $n$.
\end{example}

\begin{table}[h]
	\centering
	\resizebox{\columnwidth}{!}{%
		\renewcommand{\arraystretch}{1.3}
		\begin{tabular}{|c||c|c|c|}
			\hline
			$r$& $n=2$ & $n=3$ & $n=4$ \\
			\hline
			2 & (1,2,1,2) & (1,2,1,2,1) & (1,2,1,2,1,2) \\
			\hline
			3 & (1,2,3,1,2) + (1,3,1,2,3) & (1,2,3,1,2,3) + (1,2,1,2,3,1) & \phantom{+} (1,2,3,1,2,3,1) + (1,2,3,2,3,1,2) \\
			& + (1,2,3,2,3) & + (1,2,3,1,3,1) & + (1,2,3,1,2,1,2) + (1,3,1,2,3,1,2) \\
			& & & + (1,3,1,3,1,2,3) + (1,2,3,2,3,2,3) \\
			& & & + (1,3,1,2,3,2,3) \\
			\hline
			4 & \phantom{+} (1,2,3,4,1,2) + (1,3,4,1,2,3) & \phantom{+} (1,2,3,4,1,2,3) + (1,2,4,1,2,3,4) & \\
			& + (1,2,3,4,2,3) + (1,4,1,2,3,4) & + (1,2,3,4,1,3,4) + (1,2,1,2,3,4,1) & 25 \text{ terms } \\
			& + (1,2,4,2,3,4) + (1,2,3,4,3,4) & + (1,2,3,1,3,4,1) + (1,2,3,4,1,4,1) & \\
			\hline
		\end{tabular}
	}
	\vspace*{2pt}
	\caption{The values of $\psi_{\mathcal X}(r)(e_n)$ for small values of $r$ and $n$.}
	\label{fig: surjection Steenrod products}
\end{table}

\subsection{The $E_\infty$-prop $\mathcal M$}

We start by reviewing the finitely presented $E_\infty$-prop introduced in \cite{medina2020prop1}.
Let $\mathcal M$ be the prop generated by
\begin{equation*}
\counit \in \mathcal M(1,0)_0, \hspace*{.6cm} \coproduct \in \mathcal M(1,2)_0, \hspace*{.6cm} \product \in \mathcal M(2,1)_1,
\end{equation*}
with boundary
\begin{equation*}
\partial\ \counit = 0, \hspace*{.6cm} \partial \ \coproduct = 0, \hspace*{.6cm} \partial \ \product = \ \boundary \, ,
\end{equation*}
and restricted by the relations
\begin{equation*}
\productcounit, \hspace*{.6cm} \leftcounitality \, , \hspace*{.6cm} \rightcounitality \, .
\end{equation*}

The second named author established in \cite[Theorem 3.3]{medina2020prop1} that $\mathcal M$ is an $E_\infty$-prop as introduced in Definition~\ref{def: e-infinity operad and prop}.
More precisely, it shows that the associated operad $U_2(\mathcal M) = \big\{ \mathcal M(1, r) \big\}_{r \geq 0}$ is an $E_\infty$-operad.
For the remainder of this article we write $U(\mathcal M)$ instead of $U_2(\mathcal M)$.

We will define a May--Steenrod structure $\psi_{U(\mathcal M)} \colon \mathcal W \to U(\mathcal M)$ by composing the May--Steenrod structure $\psi_{\mathcal X} \colon \mathcal W \to \mathcal X$ with a morphism $SL \colon \mathcal X \to U(\mathcal M)$ of $\mathrm{C}$-modules ($\mathrm{S}$-modules in fact) that we now define.
Given a surjection $s \colon \{1, \dots, n+r\} \to \{1, \dots, r\}$ let $SL(s)$ be the element represented by the immersed connected $(1,r)$-graph:

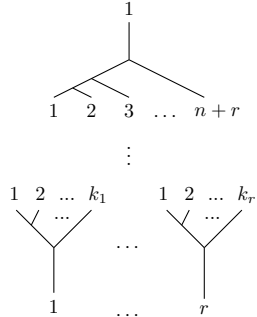
\begin{center}
	\begin{tikzpicture}[scale=1]
\draw (0,0)--(0,-.6) node[below, scale=.75]{$1$};
\draw (0,0)--(.5,.5);
\draw (-.3, .3)-- (-.2,.5) node[scale=.75] at (-.2,.7) {\qquad $1\, \ \ 2\ \ ...\ \ k_1$};
\draw (-.5,.5)--(0,0);
\node[scale=.75] at (.11,.4){$..$.};

\node[scale=.75] at (1,0){$\cdots$};
\node[scale=.75] at (1,-.9){$\cdots$};

\draw (2,0)--(2,-.68) node[scale=.75, below]{$r$};
\draw (2,0)--(2.5,.5);
\draw (1.7, .3)--(1.8,.5) node[scale=.75] at (1.78,.7) {\qquad $1\, \ \ 2\ \ ...\ \ k_r$};
\draw (1.5,.5)--(2,0);
\node[scale=.75] at (2.11,.4){$..$.};

\draw (1,2.5)--(1,3) node[scale=.75, above]{$1$};
\draw (1,2.5)--(0,2) node[scale=.75, below]{$1$};
\draw (.25,2.125)--(.5,2) node[scale=.75, below]{$2$};
\draw (.5,2.25)--(1,2) node[scale=.75, below]{$3$};
\draw (1,2.5)--(2,2) node[scale=.75, below]{\ \quad $n + r$};
\node[scale=.75] at (1.5,1.75){$\cdots$};

\node[scale=.75] at (1,1.3) {$\vdots$};

\node at (3.09,0){};
\end{tikzpicture}
\end{center}

that has no internal vertices and such that the $n+r$ strands at the top are orderly connected to the strands at the bottom following the values of $s$.

It can be directly verified using the presentation of $\mathcal M$ that the boundary of $SL(s)$ is obtained by removing strands one at a time in the order they are attached at the bottom.
This is precisely the image of $SL(\partial s)$ according to \eqref{eq: surjection boundary}.
Furthermore, relabeling the bottom edges agrees with the permutation of preimages of the associated surjection.
Since both operads have the homology of a point and $(1, \dots, r)$ is sent to a representative of a homology generator we have proven the following theorem.

\begin{theorem} \label{thm: Steenrod-Adem on U(M)}
	The composition
	\begin{equation*}
	\psi_{U(\mathcal M)} \colon \mathcal W \xra{\psi_{\mathcal X}} \mathcal X \xra{SL} U(\mathcal M)
	\end{equation*}
	defines a May--Steenrod structure on $U(\mathcal M)$.
\end{theorem}

\begin{example}
	The following immersed $(1,2)$-graphs are the elements $\psi_{U(\mathcal M)}(2)(e_n)$ for small values of $n$:
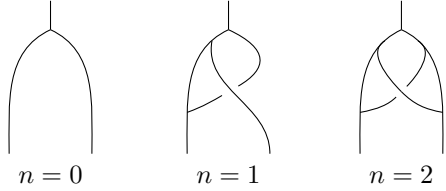
	\begin{center}
		\begin{tikzpicture}[scale=.55]
\draw (1,3.7) to (1,3);

\draw (1,3) to [out=205, in=90] (0,0);
\draw (1,3) to [out=-25, in=90] (2,0);

\node at (1,-.5){$n=0$};
\end{tikzpicture}\hspace*{1cm}
\begin{tikzpicture}[scale=.55]
\draw (1,3.7) to (1,3);

\draw (1,3) to [out=205, in=90] (0,0);

\draw [shorten >= 0cm] (.6,2.73) to [out=-100, in=90] (2,0);

\draw [shorten >= .15cm] (1,3) to [out=-25, in=30, distance=1.1cm] (1,1.5);
\draw [shorten <= .1cm] (1,1.5) to [out=210, in=20] (0,1);

\node at (1,-.5){$n=1$};
\end{tikzpicture}\hspace*{1cm}
\begin{tikzpicture}[scale=.55]
\draw (1,3.7) to (1,3);

\draw (1,3) to [out=205, in=90] (0,0);
\draw (1,3) to [out=-25, in=90] (2,0);

\draw [shorten >= 0cm] (.6,2.73) to [out=210, in=135] (1,1.5);
\draw [shorten <= 0cm] (1,1.5) to [out=-45, in=170] (2,1);

\draw [shorten >= .1cm] (1.4,2.73) to [out=-30, in=45] (1,1.5);
\draw [shorten <= .1cm] (1,1.5) to [out=-135, in=10] (0,1);

\node at (1,-.5){$n=2$};
\end{tikzpicture}
	\end{center}
\end{example}

\subsection{Cochains of simplicial sets}

In this subsection we introduce a natural May--Steenrod structure on the normalized cochains of any simplicial set $X$.
Since a May--Steenrod structure was constructed in the previous section for $U(\mathcal M)$, we only need to describe a natural $U(\mathcal M)$-algebra structure on $N^\bullet(X)$.
Using the linear duality functor, it suffices to construct a natural $U(\mathcal M)$-coalgebra structure on $N_\bullet(X)$ which, in turn, can be derived via a Kan extension argument from one on each $N_\bullet(\triangle^n)$.
We obtain these coalgebra structures by restricting a full $\mathcal M$-bialgebra structure.
An $\mathcal M$-bialgebra structure is specified by three linear maps, the images of the generators
\begin{equation*}
\counit, \quad \coproduct, \quad \product,
\end{equation*}
satisfying the relations in the presentation of $\mathcal M$.
For $n \in \mathbb{N}$, define the following: \vspace*{5pt} \\
(1) Define the counit $\epsilon \in \Hom(N_\bullet(\triangle^n), \Z)$ by
\begin{equation*}
\epsilon \big( [v_0, \dots, v_q] \big) = \begin{cases} 1 & \text{ if } q = 0, \\ 0 & \text{ if } q>0. \end{cases}
\end{equation*}
(2) Define the coproduct $\Delta \in \Hom(N_\bullet(\triangle^n), N_\bullet(\triangle^n)^{\otimes2})$ by
\begin{equation*}
\Delta \big( [v_0, \dots, v_q] \big) = \sum_{i=0}^q [v_0, \dots, v_i] \otimes [v_i, \dots, v_q].
\end{equation*}
(3) Define the product $\ast \in \Hom(N_\bullet(\triangle^n)^{\otimes 2}, N_\bullet(\triangle^n))$ by
\begin{equation*}
\left[v_0, \dots, v_p \right] \ast \left[v_{p+1}, \dots, v_q\right] = \begin{cases} (-1)^{p+|\pi|} \left[v_{\pi(0)}, \dots, v_{\pi(q)}\right] & \text{ if } v_i \neq v_j \text{ for } i \neq j, \\
0 & \text{ if not}, \end{cases}
\end{equation*}
where $\pi$ is the permutation that orders the totally ordered set of vertices, and $(-1)^{|\pi|}$ its sign.

\begin{proposition}[\cite{medina2020prop1}] \label{prop: simplicial chain bialgebra}
	For every $n \in \mathbb{N}$, the assignment
	\begin{equation*}
	\counit \mapsto \epsilon, \quad \coproduct \mapsto \Delta, \quad \product \mapsto \ast,
	\end{equation*}
	defines a natural $\mathcal M$-bialgebra structure on $N_\bullet(\triangle^n)$, and, via a Kan extension argument, a natural $U(\mathcal M)$-algebra structure $\phi^\triangle \colon U(\mathcal M) \to \End_{N^\bullet (X)}$ on the cochains of any simplicial set $X$.
\end{proposition}

Composing the algebra structure $\phi^\triangle$ with the May--Steenrod structure on $U(\mathcal M)$ gives a natural May--Steenrod structure on $N^\bullet(X)$.
We record this observation in the following.

\begin{theorem}
	The commutative diagram
	\begin{equation*}
	\begin{tikzcd}[column sep = small, row sep = small]
	&[18pt] U(\mathcal M) \arrow[dr, "\phi^\triangle", out=0] &[-0pt] \\
	\mathcal W \arrow[ur, "\psi_{U(\mathcal M)}", in=180] \arrow[rr, "\psi^\triangle"] & & \mathrm{End}_{N^\bullet (X)}
	\end{tikzcd}
	\end{equation*}
	defines a natural May--Steenrod structure on $N^\bullet(X)$ for any simplicial set $X$.
\end{theorem}

\begin{remark}
	The $E_\infty$-structure we described in Proposition~\ref{prop: simplicial chain bialgebra}, depending solely on three fundamental maps, generalizes the coalgebra structures of McClure--Smith \cite{mcclure2003multivariable} and Berger--Fresse \cite{berger2004combinatorial}, please consult \cite{medina2020prop1} for more details.
\end{remark}

We will now give examples of how this May--Steenrod structure defines representatives of Steenrod operations for simplicial cochains.
For applications related to the cohomology of spaces, it is convenient to introduce the notation $P^s = P_{-s}$ and $\beta P^s = \beta P_{-s}$ for Steenrod operations.

\begin{example}
	Let us consider the prime $2$.
	The value $P^1(x)\big([0,1,2,3,4]\big)$ for a homogeneous cocycle $x$ in $N^{-3}(\triangle^4)$ is equal to the value of $x^{\otimes 2}$ acting on
	\begin{align*}
	[0, 1, 2, 3] \otimes [0, 1, 3, 4] \ +\ &
	[0, 2, 3, 4] \otimes [0, 1, 2, 4] \\ \ +\
	[0, 1, 2, 3] \otimes [1, 2, 3, 4] \ +\ &
	[0, 1, 3, 4] \otimes [1, 2, 3, 4].
	\end{align*}
	Similarly, the value of $P^2(y)\big([0,1,2,3,4,5,6,7]\big)$ for a homogeneous cocycle $y$ in $N^{-5}(\triangle^7)$ is equal to the value of $y^{\otimes 2}$ acting on
	\begin{align*}
	[0, 1, 2, 5, 6, 7] \otimes [0, 1, 2, 3, 4, 5] & \ +\
	[0, 1, 2, 3, 6, 7] \otimes [0, 1, 3, 4, 5, 6] \\ \ +\
	[0, 1, 2, 3, 4, 7] \otimes [0, 1, 4, 5, 6, 7] & \ +\
	[0, 2, 3, 5, 6, 7] \otimes [0, 1, 2, 3, 4, 5] \\ \ +\
	[0, 2, 3, 4, 6, 7] \otimes [0, 1, 2, 4, 5, 6] & \ +\
	[0, 2, 3, 4, 5, 7] \otimes [0, 1, 2, 5, 6, 7] \\ \ +\
	[0, 3, 4, 5, 6, 7] \otimes [0, 1, 2, 3, 4, 5] & \ +\
	[0, 3, 4, 5, 6, 7] \otimes [0, 1, 2, 3, 5, 6] \\ \ +\
	[0, 3, 4, 5, 6, 7] \otimes [0, 1, 2, 3, 6, 7] & \ +\
	[0, 1, 2, 3, 6, 7] \otimes [1, 2, 3, 4, 5, 6] \\ \ +\
	[0, 1, 2, 3, 4, 7] \otimes [1, 2, 4, 5, 6, 7] & \ +\
	[0, 1, 3, 4, 6, 7] \otimes [1, 2, 3, 4, 5, 6] \\ \ +\
	[0, 1, 3, 4, 5, 7] \otimes [1, 2, 3, 5, 6, 7] & \ +\
	[0, 1, 4, 5, 6, 7] \otimes [1, 2, 3, 4, 5, 6] \\ \ +\
	[0, 1, 4, 5, 6, 7] \otimes [1, 2, 3, 4, 6, 7] & \ +\
	[0, 1, 2, 3, 4, 7] \otimes [2, 3, 4, 5, 6, 7] \\ \ +\
	[0, 1, 2, 4, 5, 7] \otimes [2, 3, 4, 5, 6, 7] & \ +\
	[0, 1, 2, 5, 6, 7] \otimes [2, 3, 4, 5, 6, 7].
	\end{align*}
\end{example}

\begin{example}
	Let us consider the prime $3$.
	The value $\beta P^1(x)\big([0,1,2,3,4,5,6,7,8]\big)$ for a homogeneous cocycle $x$ in $N^{-3}(\triangle^8)$ is equal to the value of $x^{\otimes 3}$ acting on
	\begin{align*}
	& \,-\, [0, 6, 7, 8] \otimes [0, 1, 2, 3] \otimes [3, 4, 5, 6] \,+\, [0, 1, 7, 8] \otimes [1, 2, 3, 4] \otimes [4, 5, 6, 7] \\ & \,-\, [0, 1, 2, 8] \otimes [2, 3, 4, 5] \otimes [5, 6, 7, 8].
	\end{align*}
	Similarly, the value of $P^1(y)\big([0,1,\dots,7]\big)$ for a homogeneous cocycle $y$ in $N^{-3}(\triangle^7)$ is equal to the value of $y^{\otimes 3}$ acting on
	\begin{align*}
	& \,-\, [0,3,4,5] \otimes [0,5,6,7] \otimes [0,1,2,3] \,-\, [0,4,5,6] \otimes [0,1,6,7] \otimes [1,2,3,4] \\ & \,-\, [0,5,6,7] \otimes [0,1,2,7] \otimes [2,3,4,5] \,-\, [0,1,4,5] \otimes [1,5,6,7] \otimes [1,2,3,4] \\ & \,+\, [0,1,5,6] \otimes [1,2,6,7] \otimes [2,3,4,5] \,-\, [0,1,6,7] \otimes [1,2,3,7] \otimes [3,4,5,6] \\ & \,-\, [0,1,2,5] \otimes [2,5,6,7] \otimes [2,3,4,5] \,-\, [0,1,2,6] \otimes [2,3,6,7] \otimes [3,4,5,6] \\ & \,-\, [0,1,2,7] \otimes [2,3,4,7] \otimes [4,5,6,7] \,+\, [0,1,2,3] \otimes [3,4,5,6] \otimes [0,1,6,7] \\ & \,+\, [0,2,3,4] \otimes [4,5,6,7] \otimes [0,1,2,7] \,+\, [0,1,2,3] \otimes [3,4,5,6] \otimes [1,2,6,7] \\ & \,-\, [0,1,3,4] \otimes [4,5,6,7] \otimes [1,2,3,7] \,+\, [0,1,2,3] \otimes [3,4,5,6] \otimes [2,3,6,7] \\ & \,+\, [0,1,2,4] \otimes [4,5,6,7] \otimes [2,3,4,7] \,+\, [0,1,2,3] \otimes [3,4,5,6] \otimes [3,4,6,7] \\ & \,-\, [0,1,2,3] \otimes [3,5,6,7] \otimes [3,4,5,7] \,+\, [0,1,2,3] \otimes [3,4,5,6] \otimes [4,5,6,7] \\ & \,+\, [0,1,2,3] \otimes [3,4,6,7] \otimes [4,5,6,7].
	\end{align*}
\end{example}

\subsection{Cochains of cubical sets}

In this subsection we introduce, closely following the presentation of the previous subsection, a natural May--Steenrod structure on the normalized cochains of any cubical set.
By the same considerations, the desired construction will follow from a natural $\mathcal M$-bialgebra structure on $N_\bullet(\square^n)$.
These are determined by three linear maps satisfying the relations in the presentation of $\mathcal M$.
For $n \in \mathbb{N}$, define the following: \vspace*{5pt} \\
(1) Define the counit $\epsilon \in \Hom(N_\bullet(\square^n), \Z)$ by
\begin{equation*}
\epsilon \left( x_1 \otimes \cdots \otimes x_d \right) = \epsilon(x_1) \cdots \, \epsilon(x_n),
\end{equation*}
where
\begin{equation*}
\epsilon([0]) = \epsilon([1]) = 1, \qquad \epsilon([0, 1]) = 0.
\end{equation*} \vspace*{-6pt} \\
(2) Define the coproduct $\Delta \in \Hom \left( N_\bullet(\square^n), N_\bullet(\square^n)^{\otimes 2} \right)$ by
\begin{equation*}
\Delta (x_1 \otimes \cdots \otimes x_n) =
\sum \pm \left( x_1^{(1)} \otimes \cdots \otimes x_n^{(1)} \right) \otimes
\left( x_1^{(2)} \otimes \cdots \otimes x_n^{(2)} \right),
\end{equation*}
where the sign is determined using the Koszul convention, and we are using Sweedler's notation
\begin{equation*}
\Delta(x_i) = \sum x_i^{(1)} \otimes x_i^{(2)}
\end{equation*}
for the chain map $\Delta \colon N_\bullet(\square^1) \to N_\bullet(\square^1)^{\otimes 2}$ defined by
\begin{equation*}
\Delta([0]) = [0] \otimes [0], \quad \Delta([1]) = [1] \otimes [1], \quad \Delta([0, 1]) = [0] \otimes [0, 1] + [0, 1] \otimes [1].
\end{equation*}
By using that $N_\bullet(\square^n) = N_\bullet(\square^1)^{\otimes n}$, $\Delta$ is the composition
\begin{equation*}
\begin{tikzcd}
N_\bullet(\square^1)^{\otimes n} \arrow[r, "\Delta^{\otimes n}"] & \left( N_\bullet(\square^1)^{\otimes 2} \right)^{\otimes n} \arrow[r, "sh"] & \left( N_\bullet(\square^1)^{\otimes n} \right)^{\otimes 2}
\end{tikzcd}
\end{equation*}
where $sh$ is the shuffle map that places tensor factors in odd position first. \vspace*{5pt} \\
(3) Define the product $\ast \in \Hom(N_\bullet(\square^n)^{\otimes 2}, N_\bullet(\square^n))$ by
\begin{align*}
(x_1 \otimes \cdots \otimes x_n) \ast (y_1 \otimes \cdots \otimes y_n) =
(-1)^{|x|} \sum_{i=1}^n x_{<i} \epsilon(y_{<i}) \otimes x_i \ast y_i \otimes \epsilon(x_{>i})y_{>i},
\end{align*}
where
\begin{align*}
x_{<i} & = x_1 \otimes \cdots \otimes x_{i-1}, &
y_{<i} & = y_1 \otimes \cdots \otimes y_{i-1}, \\
x_{>i} & = x_{i+1} \otimes \cdots \otimes x_n, &
y_{>i} & = y_{i+1} \otimes \cdots \otimes y_n,
\end{align*}
with the convention
\begin{equation*}
x_{<1} = y_{<1} = x_{>n} = y_{>n} = 1 \in \Z,
\end{equation*}
and the only non-zero values of $x_i \ast y_i$ are
\begin{equation*}
\ast([0] \otimes [1]) = [0, 1], \qquad \ast([1] \otimes [0]) = -[0, 1].
\end{equation*}

\begin{proposition}[\cite{medina2021cubical}] \label{prop: cubical chain bialgebra}
	For every $n \in \mathbb{N}$, the assignment
	\begin{equation*}
	\counit \mapsto \epsilon, \quad \coproduct \mapsto \Delta, \quad \product \mapsto \ast,
	\end{equation*}
	defines a natural $\mathcal M$-bialgebra structure on $N_\bullet(\square^n)$, and, via a Kan extension argument, a natural $U(\mathcal M)$-algebra structure $\phi^\square \colon U(\mathcal M) \to \End_{N^\bullet (X)}$ on the cochains of any cubical set $X$.
\end{proposition}

Composing the algebra structure $\phi^\square$ with the May--Steenrod structure on $U(\mathcal M)$ gives a natural May--Steenrod structure on $N^\bullet(X)$.
We record this observation in the following theorem.

\begin{theorem}
	The commutative diagram
	\begin{equation*}
	\begin{tikzcd}[column sep = small, row sep = small]
	&[18pt] U(\mathcal M) \arrow[dr, "\phi^\square", out=0] &[-0pt] \\
	\mathcal W \arrow[ur, "\psi_{U(\mathcal M)}", in=180] \arrow[rr, "\psi^\square"] & & \End_{N^\bullet (X)}
	\end{tikzcd}
	\end{equation*}
	defines a natural May--Steenrod structure on $N^\bullet(X)$ for any cubical set $X$.
\end{theorem}

We will now give examples of how this May--Steenrod structure defines representatives of Steenrod operations for cubical cochains.
Recall the notation $P^s = P_{-s}$ and $\beta P^s = \beta P_{-s}$ for Steenrod operations used when studying the cohomology of spaces.

\begin{example}
	Let us consider the prime $2$.
	The value $P^1(x)\big([01]^{4}\big)$ for a homogeneous cocycle $x$ in $N^{-3}(\square^4)$ is equal to the value of $x^{\otimes 2}$ acting on
	\begin{align*}&
	[01]1[01][01] \otimes [01][01]0[01]\,+\,
	[01][01][01]0 \otimes [01]0[01][01]\,+\, \\&
	[01][01][01]0 \otimes 1[01][01][01]\,+\,
	[01][01]1[01] \otimes [01]0[01][01]\,+\, \\&
	[01][01]1[01] \otimes 1[01][01][01]\,+\,
	[01][01][01]0 \otimes [01][01]1[01]\,+\, \\&
	[01][01]0[01] \otimes [01][01][01]1\,+\,
	[01]1[01][01] \otimes [01][01][01]1\,+\, \\&
	0[01][01][01] \otimes [01][01][01]1\,+\,
	0[01][01][01] \otimes [01][01]0[01]\,+\, \\&
	0[01][01][01] \otimes [01]1[01][01]\,+\,
	[01]0[01][01] \otimes 1[01][01][01].\,\phantom{+}\,
	\end{align*}
\end{example}

\begin{example}
	Let us consider the prime $3$.
	The value of $\beta P^0(x)\big([01]^2\big)$ for a homogeneous cocycle $x$ in $N^{-1}(\square^2)$ is equal to the value of $x^{\otimes 3}$ acting on
	\begin{align*}
	&\,-\, [01]1 \otimes [01]0 \otimes 1[01] \,-\, 0[01] \otimes [01]0 \otimes 1[01] \\
	&\,+\, [01]1 \otimes 0[01] \otimes [01]1 + 0[01] \otimes 0[01] \otimes [01]1.
	\end{align*}
\end{example}

\section{Outlook} \label{s:outlook}

This article looked at Steenrod operations from an algebraic viewpoint, a subject with rich geometric and combinatorial components as well.
For example, \cite{Postnikov} and \cite{medina2018prop2} independently introduced equivalent geometric representations of the cup-$i$ products in terms of stabilized arc surfaces \cite{KLP} and weighted ribbon graphs respectively.
In fact, an entire $E_\infty$-operad (prop) is constructed geometrically in this way.
We can also interpret the May--Steenrod structure in $U(\mathcal M)$ from the ``oriented surface" perspective, where the norm map $N$ of cyclic groups --~a key ingredient in Definition \ref{def: minimal cyclic resolution}~-- was identified in \cite{KLP} with a Dehn twist operator in connection with Connes' cyclic complex.

There is a functorial approach to the theory using the formalism of Feynman categories \cite{feynman}.
This includes cyclic, planar cyclic as well as Berger's pre-operads \cite{BergerRecog}.
Their interplay is of independent interest \cite{BergerKaufmann, feyrep} and will be linked directly to the constructions of this paper.

In higher category theory, the paper \cite{medina2020globular} constructs a functor producing strict $\infty$-categories from group-like cup-$i$ coalgebras in a manner similar to \cite{steiner2004omega}.
In particular, the cup-$i$ constructions described in this article for standard simplices and cubes define, respectively, the Street nerve and cubical nerve of strict $\infty$-categories.
We anticipate that the more general Steenrod cup-$(p,i)$ products constructed in this work will also have deep combinatorial interpretations.

In physics, Gaiotto, Kapustin, Thorngren \cite{gaiotto2016spin, kapustin2017fermionic, bhardwaj2017fermionic} and others have considered cellular decompositions of spacetime together with fields represented by cellular cochains.
In order to express subtle interactions between these fields, they have used cup-$i$ products to define relevant action functionals for topological field theories.
We expect that new topological field theories of interest can be studied using the Steenrod cup-$(p,i)$ products introduced in this work.

In this article we have not focused on the operations that exist non-trivially for $E_n$-algebras with $n$ finite, see part III of \cite{may76homology}.
A treatment close to ours for the $E_2$ case was given in \cite{tourtchine2006cohen}.
Effective constructions for these Dyer--Lashof--Cohen operations should be related to the geometry and combinatorial structure of configuration spaces \cite{KZhang, sinha2013littledisks, berger2004combinatorial, ayala2014configuration} and higher categories \cite{Bathigher, BalFiedSchwVogt, Rezkhigher}, and will appear elsewhere.

	\printbibliography

@book{adams1995stable,
	AUTHOR = {Adams, J. F.},
	TITLE = {Stable homotopy and generalised homology},
	SERIES = {Chicago Lectures in Mathematics},
	NOTE = {Reprint of the 1974 original},
	PUBLISHER = {University of Chicago Press, Chicago, IL},
	YEAR = {1995},
	PAGES = {x+373},
	ISBN = {0-226-00524-0},
	MRCLASS = {55-02 (57-02)},
	MRNUMBER = {1324104},
}

@book{benson1998rep,
	AUTHOR = {Benson, D. J.},
	TITLE = {Representations and cohomology. {II}},
	SERIES = {Cambridge Studies in Advanced Mathematics},
	VOLUME = {31},
	EDITION = {Second},
	NOTE = {Cohomology of groups and modules},
	PUBLISHER = {Cambridge University Press, Cambridge},
	YEAR = {1998},
	PAGES = {xii+279},
	ISBN = {0-521-63652-3},
	MRCLASS = {20-02 (19-02 20Cxx 20Jxx 55-02)},
	MRNUMBER = {1634407},
}

@book{jacobson1989algebra,
	AUTHOR = {Jacobson, Nathan},
	TITLE = {Basic algebra. {II}},
	EDITION = {Second},
	PUBLISHER = {W. H. Freeman and Company, New York},
	YEAR = {1989},
	PAGES = {xviii+686},
	ISBN = {0-7167-1933-9},
	MRCLASS = {00A05 (12-01 15-01 16-01)},
	MRNUMBER = {1009787},
	MRREVIEWER = {P. M. Cohn},
}

@article{BergerKaufmann,
	AUTHOR = {Berger, Clemens and Kaufmann, Ralph M.},
	TITLE = {Comprehensive factorisation systems},
	JOURNAL = {Tbilisi Math. J.},
	FJOURNAL = {Tbilisi Mathematical Journal},
	VOLUME = {10},
	YEAR = {2017},
	NUMBER = {3},
	PAGES = {255--277},
	ISSN = {1875-158X},
	MRCLASS = {18A25 (12F10 18A32 18D50)},
	MRNUMBER = {3742580},
	MRREVIEWER = {Seyed Naser Hosseini},
	DOI = {10.1515/tmj-2017-0112},
	URL = {https://doi-org.ezproxy.lib.purdue.edu/10.1515/tmj-2017-0112},
}

@article{feynman,
	AUTHOR = {Kaufmann, Ralph M. and Ward, Benjamin C.},
	TITLE = {Feynman categories},
	JOURNAL = {Ast\'{e}risque},
	FJOURNAL = {Ast\'{e}risque},
	NUMBER = {387},
	YEAR = {2017},
	PAGES = {vii+161},
	ISSN = {0303-1179},
	ISBN = {978-2-85626-852-7},
	MRCLASS = {18D50 (16T05 18D10 18D35 18G55 55P48 55U35 81R10)},
	MRNUMBER = {3636409},
	MRREVIEWER = {Beno\^{\i}t Fresse},
}

@article{GerstenhaberVoronov,
	AUTHOR = {Gerstenhaber, Murray and Voronov, Alexander A.},
	TITLE = {Homotopy {$G$}-algebras and moduli space operad},
	JOURNAL = {Internat. Math. Res. Notices},
	FJOURNAL = {International Mathematics Research Notices},
	YEAR = {1995},
	NUMBER = {3},
	PAGES = {141--153},
	ISSN = {1073-7928},
	MRCLASS = {18C99 (18G99 57T30 81T99)},
	MRNUMBER = {1321701},
	MRREVIEWER = {J. Stasheff},
	DOI = {10.1155/S1073792895000110},
	URL = {https://doi-org.ezproxy.lib.purdue.edu/10.1155/S1073792895000110},
}

@article{KZhang,
	AUTHOR = {Kaufmann, Ralph M. and Zhang, Yongheng},
	TITLE = {Permutohedral structures on {$E_2$}-operads},
	JOURNAL = {Forum Math.},
	FJOURNAL = {Forum Mathematicum},
	VOLUME = {29},
	YEAR = {2017},
	NUMBER = {6},
	PAGES = {1371--1411},
	ISSN = {0933-7741},
	MRCLASS = {55P48 (52B12 55P35 55R80)},
	MRNUMBER = {3719307},
	MRREVIEWER = {Beno\^{\i}t Fresse},
	DOI = {10.1515/forum-2016-0052},
	URL = {https://doi-org.ezproxy.lib.purdue.edu/10.1515/forum-2016-0052},
}

@article{KLP,
	AUTHOR = {Kaufmann, Ralph M. and Livernet, Muriel and Penner, R. C.},
	TITLE = {Arc operads and arc algebras},
	JOURNAL = {Geom. Topol.},
	FJOURNAL = {Geometry and Topology},
	VOLUME = {7},
	YEAR = {2003},
	PAGES = {511--568},
	ISSN = {1465-3060},
	MRCLASS = {18D50 (17B99 32G15)},
	MRNUMBER = {2026541},
	MRREVIEWER = {Hossein Abbaspour},
	DOI = {10.2140/gt.2003.7.511},
	URL = {https://doi-org.ezproxy.lib.purdue.edu/10.2140/gt.2003.7.511},
}

@incollection {Postnikov,
	AUTHOR = {Kaufmann, Ralph M.},
	TITLE = {Dimension vs. genus: a surface realization of the little
	{$k$}-cubes and an {$E_\infty$} operad},
	BOOKTITLE = {Algebraic topology---old and new},
	SERIES = {Banach Center Publ.},
	VOLUME = {85},
	PAGES = {241--274},
	PUBLISHER = {Polish Acad. Sci. Inst. Math., Warsaw},
	YEAR = {2009},
	MRCLASS = {55P48 (18D50)},
	MRNUMBER = {2503531},
	MRREVIEWER = {Julia Bergner},
	DOI = {10.4064/bc85-0-17},
	URL = {https://doi-org.ezproxy.lib.purdue.edu/10.4064/bc85-0-17},
}

@incollection{feyrep,
	AUTHOR = {Kaufmann, Ralph M.},
	TITLE = {{F}eynman {C}ategories and {R}epresentation {T}heory},
	SERIES = {Contemp. Math.},
	PUBLISHER = {Amer. Math. Soc., Providence, RI},
	NOTE = {to appear. ArXiv 1911.10169},
}

@incollection{sinha2013littledisks,
	AUTHOR = {Sinha, Dev P.},
	TITLE = {The (non-equivariant) homology of the little disks operad},
	BOOKTITLE = {O{PERADS} 2009},
	SERIES = {S\'{e}min. Congr.},
	VOLUME = {26},
	PAGES = {253--279},
	PUBLISHER = {Soc. Math. France, Paris},
	YEAR = {2013},
	MRCLASS = {18D50 (55P48 55R80)},
	MRNUMBER = {3203375},
	MRREVIEWER = {Benjamin C. Ward},
}

@inproceedings{BergerRecog,
	AUTHOR = {Berger, Clemens},
	TITLE = {Combinatorial models for real configuration spaces and
	{$E_n$}-operads},
	BOOKTITLE = {Operads: {P}roceedings of {R}enaissance {C}onferences
	({H}artford, {CT}/{L}uminy, 1995)},
	SERIES = {Contemp. Math.},
	VOLUME = {202},
	PAGES = {37--52},
	PUBLISHER = {Amer. Math. Soc., Providence, RI},
	YEAR = {1997},
	MRCLASS = {18D35 (06A07 18B35 20B30 55P35 55U10)},
	MRNUMBER = {1436916},
	MRREVIEWER = {Peter J. Eccles},
	DOI = {10.1090/conm/202/02582},
	URL = {https://doi-org.ezproxy.lib.purdue.edu/10.1090/conm/202/02582},
}

@article {BalFiedSchwVogt,
	AUTHOR = {Balteanu, C. and Fiedorowicz, Z. and Schw\"{a}nzl, R. and Vogt,
	R.},
	TITLE = {Iterated monoidal categories},
	JOURNAL = {Adv. Math.},
	FJOURNAL = {Advances in Mathematics},
	VOLUME = {176},
	YEAR = {2003},
	NUMBER = {2},
	PAGES = {277--349},
	ISSN = {0001-8708},
	MRCLASS = {18D10 (55P35)},
	MRNUMBER = {1982884},
	MRREVIEWER = {Pilar C. Carrasco},
	DOI = {10.1016/S0001-8708(03)00065-3},
	URL = {https://doi-org.ezproxy.lib.purdue.edu/10.1016/S0001-8708(03)00065-3},
}

@article{Bathigher,
	Author = {Batanin, M. A.},
	Doi = {10.1006/aima.1998.1724},
	Fjournal = {Advances in Mathematics},
	Issn = {0001-8708},
	Journal = {Adv. Math.},
	Mrclass = {18D05 (55P15)},
	Mrnumber = {1623672},
	Mrreviewer = {R. H. Street},
	Number = {1},
	Pages = {39--103},
	Title = {Monoidal globular categories as a natural environment for the theory of weak {$n$}-categories},
	Url = {https://doi-org.ezproxy.lib.purdue.edu/10.1006/aima.1998.1724},
	Volume = {136},
	Year = {1998}}

@article {Rezkhigher,
	AUTHOR = {Rezk, Charles},
	TITLE = {A {C}artesian presentation of weak {$n$}-categories},
	JOURNAL = {Geom. Topol.},
	FJOURNAL = {Geometry \& Topology},
	VOLUME = {14},
	YEAR = {2010},
	NUMBER = {1},
	PAGES = {521--571},
	ISSN = {1465-3060},
	MRCLASS = {18D05 (55U40)},
	MRNUMBER = {2578310},
	DOI = {10.2140/gt.2010.14.521},
	URL = {https://doi-org.ezproxy.lib.purdue.edu/10.2140/gt.2010.14.521},
}

@ARTICLE{medina2021adem,
	author={Brumfiel, Greg and {Medina-Mardones}, Anibal and Morgan, John},
	title={A cochain level proof of Adem relations in the mod 2 Steenrod algebra},
	JOURNAL = {J. Homotopy Relat. Struct.},
	FJOURNAL = {Journal of Homotopy and Related Structures},
	year={2021},
	day={19},
	issn={1512-2891},
	doi={10.1007/s40062-021-00287-3},
	url={https://doi.org/10.1007/s40062-021-00287-3}
}

@ARTICLE{medina2020prop1,
	AUTHOR = {{Medina-Mardones}, Anibal M.},
	TITLE = {A finitely presented {$E_\infty$}-prop {I}: algebraic context},
	JOURNAL = {High. Struct.},
	FJOURNAL = {Higher Structures},
	VOLUME = {4},
	YEAR = {2020},
	NUMBER = {2},
	PAGES = {1--21},
	MRCLASS = {55P48 (18M85 18N50)},
	MRNUMBER = {4133162},
	url = {https://journals.mq.edu.au/index.php/higher_structures/article/view/95}
}

@ARTICLE{medina2020globular,
	AUTHOR = {{Medina-Mardones}, Anibal M.},
	TITLE = {An algebraic representation of globular sets},
	JOURNAL = {Homology Homotopy Appl.},
	FJOURNAL = {Homology, Homotopy and Applications},
	VOLUME = {22},
	YEAR = {2020},
	NUMBER = {2},
	PAGES = {135--150},
	ISSN = {1532-0073},
	MRCLASS = {18N99 (18F99 55S05)},
	MRNUMBER = {4093174},
	DOI = {10.4310/hha.2020.v22.n2.a8},
	URL = {https://doi.org/10.4310/hha.2020.v22.n2.a8},
}

@ARTICLE{medina2020cartan,
	AUTHOR = {{Medina-Mardones}, Anibal M.},
	TITLE = {An effective proof of the {C}artan formula: the even prime},
	JOURNAL = {J. Pure Appl. Algebra},
	FJOURNAL = {Journal of Pure and Applied Algebra},
	VOLUME = {224},
	YEAR = {2020},
	NUMBER = {12},
	PAGES = {106444, 18},
	ISSN = {0022-4049},
	MRCLASS = {55S10 (55S05 55S12)},
	MRNUMBER = {4102178},
	DOI = {10.1016/j.jpaa.2020.106444},
	URL = {https://doi.org/10.1016/j.jpaa.2020.106444},
}

@ARTICLE{medina2021computer,
	author = {{Medina-Mardones}, Anibal M.},
	title = "{A computer algebra system for the study of commutativity up to coherent homotopies}",
	journal = {arXiv e-prints},
	year = 2021,
	archivePrefix = {arXiv},
	eprint = {2102.07670},
	note = {To appear in Tbilisi Math. J.}
}

@ARTICLE{medina2018prop2,
	author = {{Medina-Mardones}, Anibal M.},
	title = "{A finitely presented ${E}_{\infty}$-prop II: cellular context}",
	journal = {arXiv e-prints},
	year = 2018,
	archivePrefix = {arXiv},
	eprint = {1808.07132},
	note = {To appear in High. Struct.}
}

@ARTICLE{medina2021cobar,
	author = {{Medina-Mardones}, Anibal M. and {Rivera}, Manuel},
	title = "{The cobar construction as an $E_{\infty}$-bialgebra model of the based loop space}",
	journal = {arXiv e-prints},
	year = 2021,
	archivePrefix = {arXiv},
	eprint = {2108.02790},
}

@ARTICLE{medina2021cubical,
	author = {{Kaufmann}, Ralph M. and {Medina-Mardones}, Anibal M.},
	title = "{A combinatorial ${E}_\infty$-algebra structure on cubical cochains}",
	journal = {arXiv e-prints},
	year = 2021,
	archivePrefix = {arXiv},
	eprint = {2107.00669}
}

@ARTICLE{medina2021newformulas,
	author = {{Medina-Mardones}, Anibal M.},
	title = "{New formulas for cup-$i$ products and fast computation of Steenrod squares}",
	journal = {arXiv e-prints},
	year = 2021,
	archivePrefix = {arXiv},
	eprint = {2105.08025}
}

@ARTICLE{medina2018persistence,
	author = {{Medina-Mardones}, Anibal M.},
	title = "{Persistence Steenrod modules}",
	journal = {arXiv e-prints},
	year = 2018,
	archivePrefix = {arXiv},
	eprint = {1812.05031},
}

@book {adem2004milgram,
	AUTHOR = {Adem, Alejandro and Milgram, R. James},
	TITLE = {Cohomology of finite groups},
	SERIES = {Grundlehren der Mathematischen Wissenschaften [Fundamental
	Principles of Mathematical Sciences]},
	VOLUME = {309},
	EDITION = {Second},
	PUBLISHER = {Springer-Verlag, Berlin},
	YEAR = {2004},
	PAGES = {viii+324},
	ISBN = {3-540-20283-8},
	MRCLASS = {20J06 (18G99 55R35 55R40)},
	MRNUMBER = {2035696},
	DOI = {10.1007/978-3-662-06280-7},
	URL = {https://doi.org/10.1007/978-3-662-06280-7},
}

@book {boardman1973homotopy,
	AUTHOR = {Boardman, J. M. and Vogt, R. M.},
	TITLE = {Homotopy invariant algebraic structures on topological spaces},
	SERIES = {Lecture Notes in Mathematics, Vol. 347},
	PUBLISHER = {Springer-Verlag, Berlin-New York},
	YEAR = {1973},
	PAGES = {x+257},
	MRCLASS = {55D35},
	MRNUMBER = {0420609},
	MRREVIEWER = {J. Stasheff},
	url = {https://www.springer.com/gp/book/9783540064794}
}

@book {may1972geometry,
	AUTHOR = {May, J. P.},
	TITLE = {The geometry of iterated loop spaces},
	SERIES = {Lecture Notes in Mathematics, Vol. 271},
	PUBLISHER = {Springer-Verlag, Berlin-New York},
	YEAR = {1972},
	PAGES = {viii+175},
	MRCLASS = {55D35},
	MRNUMBER = {0420610},
	MRREVIEWER = {J. Stasheff},
	url = {https://doi.org/10.1007/BFb0067491}
}

@article {mcclure2003multivariable,
	AUTHOR = {McClure, James E. and Smith, Jeffrey H.},
	TITLE = {Multivariable cochain operations and little {$n$}-cubes},
	JOURNAL = {J. Amer. Math. Soc.},
	FJOURNAL = {Journal of the American Mathematical Society},
	VOLUME = {16},
	YEAR = {2003},
	NUMBER = {3},
	PAGES = {681--704},
	ISSN = {0894-0347},
	MRCLASS = {55P48 (18D50)},
	MRNUMBER = {1969208},
	MRREVIEWER = {Benoit Fresse},
	DOI = {10.1090/S0894-0347-03-00419-3},
	URL = {https://doi.org/10.1090/S0894-0347-03-00419-3},
}

@article {berger2004combinatorial,
	AUTHOR = {Berger, Clemens and Fresse, Benoit},
	TITLE = {Combinatorial operad actions on cochains},
	JOURNAL = {Math. Proc. Cambridge Philos. Soc.},
	FJOURNAL = {Mathematical Proceedings of the Cambridge Philosophical
	Society},
	VOLUME = {137},
	YEAR = {2004},
	NUMBER = {1},
	PAGES = {135--174},
	ISSN = {0305-0041},
	MRCLASS = {18D50 (16E45 55P48)},
	MRNUMBER = {2075046},
	MRREVIEWER = {David Chataur},
	DOI = {10.1017/S0305004103007138},
	URL = {https://doi.org/10.1017/S0305004103007138},
}

@article {chataur2005adem-cartan,
	AUTHOR = {Chataur, David and Livernet, Muriel},
	TITLE = {Adem-{C}artan operads},
	JOURNAL = {Comm. Algebra},
	FJOURNAL = {Communications in Algebra},
	VOLUME = {33},
	YEAR = {2005},
	NUMBER = {11},
	PAGES = {4337--4360},
	ISSN = {0092-7872},
	MRCLASS = {55S10 (18D50 55P48)},
	MRNUMBER = {2184004},
	MRREVIEWER = {Vincent Giambalvo},
	DOI = {10.1080/00927870500243205},
	URL = {https://doi.org/10.1080/00927870500243205},
}

@incollection {markl2008props,
	AUTHOR = {Markl, Martin},
	TITLE = {Operads and {PROP}s},
	BOOKTITLE = {Handbook of algebra. {V}ol. 5},
	SERIES = {Handb. Algebr.},
	VOLUME = {5},
	PAGES = {87--140},
	PUBLISHER = {Elsevier/North-Holland, Amsterdam},
	YEAR = {2008},
	MRCLASS = {18D50},
	MRNUMBER = {2523450},
	DOI = {10.1016/S1570-7954(07)05002-4},
	URL = {https://doi.org/10.1016/S1570-7954(07)05002-4},
}

@article {steenrod1947products,
	AUTHOR = {Steenrod, N. E.},
	TITLE = {Products of cocycles and extensions of mappings},
	JOURNAL = {Ann. of Math. (2)},
	FJOURNAL = {Annals of Mathematics. Second Series},
	VOLUME = {48},
	YEAR = {1947},
	PAGES = {290--320},
	ISSN = {0003-486X},
	MRCLASS = {56.0X},
	MRNUMBER = {22071},
	MRREVIEWER = {B. Eckmann},
	DOI = {10.2307/1969172},
	URL = {https://doi.org/10.2307/1969172},
}

@article {adem1952iteration,
	AUTHOR = {Adem, Jos\'{e}},
	TITLE = {The iteration of the {S}teenrod squares in algebraic topology},
	JOURNAL = {Proc. Nat. Acad. Sci. U.S.A.},
	FJOURNAL = {Proceedings of the National Academy of Sciences of the United
	States of America},
	VOLUME = {38},
	YEAR = {1952},
	PAGES = {720--726},
	ISSN = {0027-8424},
	MRCLASS = {56.0X},
	MRNUMBER = {50278},
	MRREVIEWER = {W. S. Massey},
	DOI = {10.1073/pnas.38.8.720},
	URL = {https://doi.org/10.1073/pnas.38.8.720},
}

@article {steenrod1952reduced,
	AUTHOR = {Steenrod, N. E.},
	TITLE = {Reduced powers of cohomology classes},
	JOURNAL = {Ann. of Math. (2)},
	FJOURNAL = {Annals of Mathematics. Second Series},
	VOLUME = {56},
	YEAR = {1952},
	PAGES = {47--67},
	ISSN = {0003-486X},
	MRCLASS = {56.0X},
	MRNUMBER = {48026},
	MRREVIEWER = {H. Cartan},
	DOI = {10.2307/1969766},
	URL = {https://doi.org/10.2307/1969766},
}

@article {steenrod1953cyclic,
	AUTHOR = {Steenrod, N. E.},
	TITLE = {Cyclic reduced powers of cohomology classes},
	JOURNAL = {Proc. Nat. Acad. Sci. U.S.A.},
	FJOURNAL = {Proceedings of the National Academy of Sciences of the United
	States of America},
	VOLUME = {39},
	YEAR = {1953},
	PAGES = {217--223},
	ISSN = {0027-8424},
	MRCLASS = {56.0X},
	MRNUMBER = {54965},
	MRREVIEWER = {H. Cartan},
	DOI = {10.1073/pnas.39.3.217},
	URL = {https://doi.org/10.1073/pnas.39.3.217},
}

@article {milnor1958dual,
	AUTHOR = {Milnor, John},
	TITLE = {The {S}teenrod algebra and its dual},
	JOURNAL = {Ann. of Math. (2)},
	FJOURNAL = {Annals of Mathematics. Second Series},
	VOLUME = {67},
	YEAR = {1958},
	PAGES = {150--171},
	ISSN = {0003-486X},
	MRCLASS = {55.00 (18.00)},
	MRNUMBER = {99653},
	MRREVIEWER = {G. Hirsch},
	DOI = {10.2307/1969932},
	URL = {https://doi.org/10.2307/1969932},
}

@book {steenrod1962cohomology,
	ISBN = {9780691079240},
	author = {Steenrod, {N. E.} and Epstein, {D. B. A.}},
	publisher = {Princeton University Press},
	title = {Cohomology Operations: Lectures by N. E. Steenrod.},
	year = {1962},
	URL = {http://www.jstor.org/stable/j.ctt1b7x52h},
}

@inproceedings {may1970general,
	AUTHOR = {May, J. Peter},
	TITLE = {A general algebraic approach to {S}teenrod operations},
	BOOKTITLE = {The {S}teenrod {A}lgebra and its {A}pplications ({P}roc.
	{C}onf. to {C}elebrate {N}. {E}. {S}teenrod's {S}ixtieth
	{B}irthday, {B}attelle {M}emorial {I}nst., {C}olumbus, {O}hio,
	1970)},
	SERIES = {Lecture Notes in Mathematics, Vol. 168},
	PAGES = {153--231},
	PUBLISHER = {Springer, Berlin},
	YEAR = {1970},
	MRCLASS = {55.34 (18.00)},
	MRNUMBER = {0281196},
	MRREVIEWER = {W. D. Barcus},
	url = {https://link.springer.com/chapter/10.1007/BFb0058524}
}

@article {gonzalez2005cocyclic,
	AUTHOR = {{Gonzalez-Diaz}, Rocio and Real, Pedro},
	TITLE = {{HPT} and cocyclic operations},
	JOURNAL = {Homology Homotopy Appl.},
	FJOURNAL = {Homology, Homotopy and Applications},
	VOLUME = {7},
	YEAR = {2005},
	NUMBER = {2},
	PAGES = {95--108},
	ISSN = {1532-0081},
	MRCLASS = {55S10 (05E99)},
	MRNUMBER = {2156309},
	MRREVIEWER = {Adriana Ciampella},
	URL = {http://projecteuclid.org/euclid.hha/1139839376},
}

@article {tourtchine2006cohen,
	AUTHOR = {Tourtchine, Victor},
	TITLE = {Dyer-{L}ashof-{C}ohen operations in {H}ochschild cohomology},
	JOURNAL = {Algebr. Geom. Topol.},
	FJOURNAL = {Algebraic \& Geometric Topology},
	VOLUME = {6},
	YEAR = {2006},
	PAGES = {875--894},
	ISSN = {1472-2747},
	MRCLASS = {55S12 (16E40 18D50 55P48)},
	MRNUMBER = {2240919},
	MRREVIEWER = {Vigleik Angeltveit},
	DOI = {10.2140/agt.2006.6.875},
	URL = {https://doi-org.ezproxy.lib.purdue.edu/10.2140/agt.2006.6.875},
}

@incollection {pilarczyk2016cubical,
	AUTHOR = {Kr\v{c}\'{a}l, Marek and Pilarczyk, Pawe{\l} },
	TITLE = {Computation of cubical {S}teenrod squares},
	BOOKTITLE = {Computational topology in image context},
	SERIES = {Lecture Notes in Comput. Sci.},
	VOLUME = {9667},
	PAGES = {140--151},
	PUBLISHER = {Springer},
	YEAR = {2016},
	MRCLASS = {55N45 (55U15 68W30)},
	MRNUMBER = {3533883},
	DOI = {10.1007/978-3-319-39441-1\_13},
	URL = {https://doi.org/10.1007/978-3-319-39441-1_13},
}

@article {araki56squaring,
	AUTHOR = {Kudo, Tatsuji and Araki, Sh\^{o}r\^{o}},
	TITLE = {Topology of {$H_n$}-spaces and {$H$}-squaring operations},
	JOURNAL = {Mem. Fac. Sci. Ky\={u}sy\={u} Univ. A},
	FJOURNAL = {Memoirs of the Faculty of Science. Kyushu University. Series
	A. Mathematics},
	VOLUME = {10},
	YEAR = {1956},
	PAGES = {85--120},
	ISSN = {0373-6385},
	MRCLASS = {55.0X},
	MRNUMBER = {87948},
	MRREVIEWER = {R. Bott},
	doi = {10.2206/kyushumfs.10.85},
	url = {https://doi.org/10.2206/kyushumfs.10.85}
}

@article {dyer62lashof,
	AUTHOR = {Dyer, Eldon and Lashof, R. K.},
	TITLE = {Homology of iterated loop spaces},
	JOURNAL = {Amer. J. Math.},
	FJOURNAL = {American Journal of Mathematics},
	VOLUME = {84},
	YEAR = {1962},
	PAGES = {35--88},
	ISSN = {0002-9327},
	MRCLASS = {55.30 (55.40)},
	MRNUMBER = {141112},
	MRREVIEWER = {J. F. Adams},
	DOI = {10.2307/2372804},
	URL = {https://doi.org/10.2307/2372804},
}

@article {barratt1972priddyquillen,
	AUTHOR = {Barratt, Michael and Priddy, Stewart},
	TITLE = {On the homology of non-connected monoids and their associated
	groups},
	JOURNAL = {Comment. Math. Helv.},
	FJOURNAL = {Commentarii Mathematici Helvetici},
	VOLUME = {47},
	YEAR = {1972},
	PAGES = {1--14},
	ISSN = {0010-2571},
	MRCLASS = {18H10 (57F05)},
	MRNUMBER = {314940},
	MRREVIEWER = {J. C. Becker},
	DOI = {10.1007/BF02566785},
	URL = {https://doi.org/10.1007/BF02566785},
}

@book {may76homology,
	AUTHOR = {Cohen, Frederick R. and Lada, Thomas J. and May, J. Peter},
	TITLE = {The homology of iterated loop spaces},
	SERIES = {Lecture Notes in Mathematics, Vol. 533},
	PUBLISHER = {Springer-Verlag, Berlin-New York},
	YEAR = {1976},
	PAGES = {vii+490},
	MRCLASS = {55G25 (55D35)},
	MRNUMBER = {0436146},
	MRREVIEWER = {Peter J. Eccles},
	url = {https://doi.org/10.1007/BFb0080464}
}

@incollection {lawson2020dyerlashof,
	AUTHOR = {Lawson, Tyler},
	TITLE = {{$E_n$}-spectra and {D}yer-{L}ashof operations},
	BOOKTITLE = {Handbook of homotopy theory},
	SERIES = {CRC Press/Chapman Hall Handb. Math. Ser.},
	PAGES = {793--849},
	PUBLISHER = {CRC Press, Boca Raton, FL},
	YEAR = {2020},
	MRCLASS = {55P43 (55S12 55S20)},
	MRNUMBER = {4197999},
	url = {https://arxiv.org/abs/2002.03889}
}

@article {adams1956cobar,
	AUTHOR = {Adams, J. F.},
	TITLE = {On the cobar construction},
	JOURNAL = {Proc. Nat. Acad. Sci. U.S.A.},
	FJOURNAL = {Proceedings of the National Academy of Sciences of the United
	States of America},
	VOLUME = {42},
	YEAR = {1956},
	PAGES = {409--412},
	ISSN = {0027-8424},
	MRCLASS = {55.0X},
	MRNUMBER = {79266},
	MRREVIEWER = {W. S. Massey},
	DOI = {10.1073/pnas.42.7.409},
	URL = {https://doi.org/10.1073/pnas.42.7.409},
}

@article {baues1998hopf,
	AUTHOR = {Baues, Hans-Joachim},
	TITLE = {The cobar construction as a {H}opf algebra},
	JOURNAL = {Invent. Math.},
	FJOURNAL = {Inventiones Mathematicae},
	VOLUME = {132},
	YEAR = {1998},
	NUMBER = {3},
	PAGES = {467--489},
	ISSN = {0020-9910},
	MRCLASS = {55P35 (18G50 55P62)},
	MRNUMBER = {1625728},
	MRREVIEWER = {Paul G. Goerss},
	DOI = {10.1007/s002220050231},
	URL = {https://doi.org/10.1007/s002220050231},
}

@article {kadeishvili2003cupi,
	AUTHOR = {Kadeishvili, T.},
	TITLE = {Cochain operations defining {S}teenrod {$\smile_i$}-products
	in the bar construction},
	JOURNAL = {Georgian Math. J.},
	FJOURNAL = {Georgian Mathematical Journal},
	VOLUME = {10},
	YEAR = {2003},
	NUMBER = {1},
	PAGES = {115--125},
	ISSN = {1072-947X},
	MRCLASS = {55P48 (55S10 57T30)},
	MRNUMBER = {1990691},
	MRREVIEWER = {Dai Tamaki},
	url = {https://www.emis.de/journals/GMJ/vol10/v10n1-9.pdf}
}

@article {edelsbrunner2002topological,
	title={Topological Persistence and Simplification},
	author={Edelsbrunner, Herbert and Letscher, David and Zomorodian, Afra},
	journal={Discrete Comput Geom},
	volume={28},
	pages={511--533},
	year={2002},
	publisher={Citeseer},
	url = {https://doi.org/10.1007/s00454-002-2885-2}
}

@article {carlsson2005barcode,
	title={Computing persistent homology},
	author={Zomorodian, Afra and Carlsson, Gunnar},
	journal={Discrete \& Computational Geometry},
	volume={33},
	number={2},
	pages={249--274},
	year={2005},
	publisher={Springer},
	url = {https://doi.org/10.1007/s00454-004-1146-y}
}

@article {chan2013viral,
	author = {Chan, Joseph Minhow and Carlsson, Gunnar and Rabadan, Raul},
	title = {Topology of viral evolution},
	volume = {110},
	number = {46},
	pages = {18566--18571},
	year = {2013},
	doi = {10.1073/pnas.1313480110},
	publisher = {National Academy of Sciences},
	issn = {0027-8424},
	URL = {https://www.pnas.org/content/110/46/18566},
	eprint = {https://www.pnas.org/content/110/46/18566.full.pdf},
	journal = {Proceedings of the National Academy of Sciences}
}

@article {hess2017cliques,
	title={Cliques of neurons bound into cavities provide a missing link between structure and function},
	author={Reimann, Michael W and Nolte, Max and Scolamiero, Martina and Turner, Katharine and Perin, Rodrigo and Chindemi, Giuseppe and D{\l}otko, Pawe{\l} and Levi, Ran and Hess, Kathryn and Markram, Henry},
	journal={Frontiers in computational neuroscience},
	volume={11},
	pages={48},
	year={2017},
	publisher={Frontiers},
	doi = {10.3389/fncom.2017.00048},
	URL = {https://doi.org/10.3389/fncom.2017.00048}
}

@article {steiner2004omega,
	AUTHOR = {Steiner, Richard},
	TITLE = {Omega-categories and chain complexes},
	JOURNAL = {Homology Homotopy Appl.},
	FJOURNAL = {Homology, Homotopy and Applications},
	VOLUME = {6},
	YEAR = {2004},
	NUMBER = {1},
	PAGES = {175--200},
	ISSN = {1532-0081},
	MRCLASS = {18D05},
	MRNUMBER = {2061574},
	MRREVIEWER = {Graham J. Ellis},
	URL = {http://projecteuclid.org/euclid.hha/1139839551},
}

@article {ayala2014configuration,
	AUTHOR = {Ayala, David and Hepworth, Richard},
	TITLE = {Configuration spaces and {$\Theta_n$}},
	JOURNAL = {Proc. Amer. Math. Soc.},
	FJOURNAL = {Proceedings of the American Mathematical Society},
	VOLUME = {142},
	YEAR = {2014},
	NUMBER = {7},
	PAGES = {2243--2254},
	ISSN = {0002-9939},
	MRCLASS = {18D05 (55R80)},
	MRNUMBER = {3195750},
	MRREVIEWER = {Jeffrey Giansiracusa},
	DOI = {10.1090/S0002-9939-2014-11946-0},
	URL = {https://doi.org/10.1090/S0002-9939-2014-11946-0},
}

@article {gaiotto2016spin,
	author = {Gaiotto, Davide and Kapustin, Anton},
	title = {Spin TQFTs and fermionic phases of matter},
	journal = {International Journal of Modern Physics A},
	volume = {31},
	number = {28n29},
	pages = {1645044},
	year = {2016},
	doi = {10.1142/S0217751X16450445},
	URL = {https://doi.org/10.1142/S0217751X16450445},
	eprint = {https://doi.org/10.1142/S0217751X16450445},
}

@article {bhardwaj2017fermionic,
	author={Bhardwaj, Lakshya
	and Gaiotto, Davide
	and Kapustin, Anton},
	title={State sum constructions of spin-TFTs and string net constructions of fermionic phases of matter},
	journal={Journal of High Energy Physics},
	year={2017},
	month={04},
	day={18},
	volume={2017},
	number={4},
	pages={96},
	issn={1029-8479},
	doi={10.1007/JHEP04(2017)096},
	url={https://doi.org/10.1007/JHEP04(2017)096}
}

@article {kapustin2017fermionic,
	AUTHOR = {Kapustin, Anton and Thorngren, Ryan},
	TITLE = {Fermionic {SPT} phases in higher dimensions and bosonization},
	JOURNAL = {J. High Energy Phys.},
	FJOURNAL = {Journal of High Energy Physics},
	YEAR = {2017},
	NUMBER = {10},
	PAGES = {080, front matter+48},
	ISSN = {1126-6708},
	MRCLASS = {81T30},
	MRNUMBER = {3731133},
	DOI = {10.1007/jhep10(2017)080},
	URL = {https://doi.org/10.1007/jhep10(2017)080},
}

@article {brumfiel2016pontrjagin,
	author = {{Brumfiel}, Greg and {Morgan}, John},
	title = "{The {P}ontrjagin Dual of 3-Dimensional Spin Bordism}",
	journal = {arXiv e-prints},
	year = 2016,
	archivePrefix = {arXiv},
	eprint = {1612.02860},
}

@article {brumfiel2018pontrjagin,
	author = {{Brumfiel}, Greg and {Morgan}, John},
	title = "{The {P}ontrjagin Dual of 4-Dimensional Spin Bordism}",
	journal = {arXiv e-prints},
	year = 2018,
	archivePrefix = {arXiv},
	eprint = {1803.08147},
}
\end{document}